\documentclass[a4paper,11pt]{amsart}

\usepackage{mathrsfs,nccmath,amssymb,amsthm,mathabx}
\usepackage{color}
\usepackage{ulem}

\newcommand{\myforall}{\text{ for all }}
\newcommand{\mysuchthat}{\text{ such that }}

\newcommand{\myand}{\text{ and }}

\newcommand{\seb}{\{\,}
\newcommand{\sen}{\,\}}

\newcommand{\getsby}[1]{\xleftarrow{#1}}

\newcommand{\Per}{{\rm Per}}

\newcommand{\funcdecomp}[2]{{#1\/}^{\/#2}}

\newcommand{\tesgh}{edge-surjective graph homomorphism}

\newcommand{\pdirectional}{\raise0.05em\hbox{$+$}directional}

\newcommand{\Z}{\mathbb{Z}}
\newcommand{\Nonne}{\mathbb{N}}
\newcommand{\Posint}{\Nonne^+}
\newcommand{\Real}{\mathbb{R}}
\newcommand{\bi}{\in \Z}

\newcommand{\beposint}{\in \Posint}
\newcommand{\bpi}{\beposint} 

\newcommand{\bni}{\in \Nonne} 

\newcommand{\diam}{{\rm diam}}
\newcommand{\mesh}{{\rm mesh}}
\newcommand{\dist}{{\rm dist}}

\newcommand{\pstrz}[2]{({#1}_0,{#1}_1,\dotsc,{#1}_{#2})}
\newcommand{\pstro}[2]{({#1}_1,{#1}_2,\dotsc,{#1}_{#2})}
\newcommand{\pstrzinf}[1]{(#1_0,#1_1,#1_2,\dotsc)}

\newcommand{\Decomp}{\mathscr{D}}
\newcommand{\Acal}{\mathcal{A}}
\newcommand{\Bcal}{\mathcal{B}}
\newcommand{\Ccal}{\mathcal{C}} 
\newcommand{\Dcal}{\mathcal{D}} 

\newcommand{\Gcal}{\mathcal{G}}

\newcommand{\Ical}{\mathcal{I}}

\newcommand{\Mcal}{\mathcal{M}}

\newcommand{\Pcal}{\mathcal{P}}

\newcommand{\Ucal}{\mathcal{U}}
\newcommand{\Vcal}{\mathcal{V}}

\newcommand{\sC}{\mathscr{C}}

\newcommand{\sM}{\mathscr{M}}

\newcommand{\sO}{\mathscr{O}}
\newcommand{\sP}{\mathscr{P}}

\newcommand{\kuu}{\emptyset}
\newcommand{\nekuu}{\neq \kuu}
\newcommand{\fai}{\varphi}

\newcommand{\bN}{{\bf N}}

\newcommand{\barm}{\bar{m}}

\newcommand{\barM}{\bar{M}}

\newcommand{\enumb}{\begin{enumerate}}
\newcommand{\enumn}{\end{enumerate}}

\newcommand{\itemb}{\begin{itemize}}
\newcommand{\itemn}{\end{itemize}}

\newtheorem{thm}{Theorem}[section]
\newtheorem{lem}[thm]{Lemma}
\newtheorem{prop}[thm]{Proposition}
\newtheorem{cor}[thm]{Corollary}

\theoremstyle{definition}
\newtheorem{defn}[thm]{Definition}

\newtheorem{example}[thm]{Example}

\theoremstyle{remark}
\newtheorem{nota}[thm]{Notation}
\newtheorem{rem}[thm]{Remark}

\numberwithin{equation}{section}

\newcommand{\usualnotationss}{(N-1) -- (N-5) }
\newcommand{\Eq}{{\rm Eq}}
\newcommand{\Icalinf}{\Ical_{\infty}}
\newcommand{\Deltainf}{\Delta_{\infty}}

\newcommand{\Ccalinf}{\Ccal_{\infty}}
\newcommand{\ov}{\overline}
\newcommand{\ovCcalinf}{\ov{\Ccal}_{\infty}}

\newcommand{\tila}{\tilde{a}}
\newcommand{\tilb}{\tilde{b}}
\newcommand{\tilc}{\tilde{c}}

\newcommand{\Pmax}{P_{\rm max}}
\newcommand{\Hull}{{\rm Hull}}
\newcommand{\abs}[1]{\lvert#1\rvert}

\begin{document}

\title[Graph Covers and Ergodicity]
{Graph Covers and Ergodicity for 0-dimensional Systems}

\author{TAKASHI SHIMOMURA}

\address{Nagoya University of Economics, Uchikubo 61-1, Inuyama 484-8504, Japan}
\curraddr{}
\email{tkshimo@nagoya-ku.ac.jp}
\thanks{}



\date{\today}

\dedicatory{}

\commby{}

\begin{abstract}
Bratteli--Vershik systems have been widely studied.
In the context of general 0-dimensional systems, Bratteli--Vershik systems are homeomorphisms that have Kakutani--Rohlin refinements.
Bratteli diagram has a strong power to analyze such systems.
Besides this approach, general graph covers can be used to represent any 0-dimensional system.
Indeed, all 0-dimensional systems can be described as a certain kind of sequences of graph covers that may not be brought about by the Kakutani--Rohlin partitions.
In this paper, we follow the context of general graph covers to
 analyze the relations between ergodic measures and circuits of graph covers.
First, we formalize the condition for a sequence of graph covers to represent minimal Cantor systems.
In constructing invariant measures, 
 we deal with general compact metrizable 0-dimensional systems.
In the context of Bratteli diagrams with finite rank, it has previously been mentioned that all ergodic measures should be limits of some combinations of towers of Kakutani--Rohlin refinements.
We demonstrate this for general 0-dimensional case, and develop a theorem that expresses the coincidence of the time average and the space average for ergodic measures.
Additionally, we formulate a theorem that signifies the old relation between uniform convergence and
 unique ergodicity in the context of graph circuits for general 0-dimensional systems.
Unlike previous studies, in our case of general graph covers,
 there arise a possibility of linear dependence of circuits.
We give a condition for a full circuit system to be linearly independent.
Previous researches also showed that
 the bounded combinatorics imply unique ergodicity.
We present a lemma that enables us to consider unbounded ranks of winding matrices.
\end{abstract}

\maketitle
%
%
%
%
\section{Introduction}\label{sec:introduction}
Let $X$ be a compact metrizable 0-dimensional space,
 and $f : X \to X$ be a continuous surjective map.
In this paper, we call $(X,f)$ a 0-dimensional system.
If $X$ is homeomorphic to the Cantor set, $(X,f)$ is called
 a Cantor system.
In \cite[\S 2]{GM}, Gambaudo and Martens defined special sequences of graph covers for Cantor minimal systems.
In \cite[Theorem 2.5]{GM}, they stated a structure theorem whereby every minimal Cantor system can be described as the inverse limit of such a sequence.
In the usage of this theorem, they combinatorially constructed 
 the set of all invariant Borel measures \cite[Proposition 3.2]{GM}.
Further, in \S 3 of the above paper, in the study of constructing a Cantor system that has the space of ergodic measures that is homeomorphic to $S^n$,
they showed how their graph covers are related to
ergodic measures.
In particular, they stated that some  combinatorics bring about finite ergodicity,
 and that bounded combinatorics bring about unique ergodicity~\cite[Proposition 3.3]{GM}.
After the first version of this manuscript, the referee notified the author
 of the widespread study of Bratteli--Vershik systems.
Herman, Putnam and Skau \cite{HPS}
 developed a method of representing any 0-dimensional systems that are homeomorphisms and
 have unique minimal sets, 
 using essentially simple ordered Bratteli diagrams.
They used a special cover, called the Kakutani--Rohlin partition.
From this, numerous studies of
 Bratteli--Vershik systems began.
For example, the Bratteli--Vershik system was extended by
 Medynets \cite{Medynets} to 0-dimensional systems that are homeomorphisms 
 and have no periodic points.
In \cite{BKMS}, Bezuglyi, Kwiatkowski, Medynets and Solomyak
 presented a number of results related to those given in this paper, as well as others in the context of
 Bratteli diagrams of finite rank.
As for finite ergodicity, \cite[Theorem 3.3 (3)]{BKMS} also refers.
And, as for the unique ergodicity, \cite[Theorem 4.11]{BKMS} gives a necessary and sufficient condition in the context of Bratteli diagrams of finite rank.
Following Akin, Glasner and Weiss \cite{AGW},
 we used general graph covers to derive \cite[Theorem \ref{thm:structure}]{Shimomura4},
 in which we showed that
 all 0-dimensional systems are naturally described
 as an inverse limit of an inverse system of directed graphs.
Unlike the use of Bratteli diagrams, our representation of general 0-dimensional systems
 uses a sequence of finite directed graphs and their covers, and 
 naturally represents every 0-dimensional system.
In particular, the natural partition induced from each directed graph is 
 not neccessarily a Kakutani--Rohlin partition.
In this paper, we  follow the results of Gambaudo and Martens \cite{GM}
 up to \S 3 in the context of our general graph covers.
And, also refer the linkage with the results of Bezuglyi, Kwiatkowski, Medynets and Solomyak \cite{BKMS} as possible as the author can.
To extend the work of Gambaudo and Martens \cite{GM} in such a way, we first
 describe the condition for a general Cantor system to be minimal in
 Theorem \ref{thm:minimal-main}.
Our construction of invariant measures applies to
 general 0-dimensional systems, and not only minimal systems
 (see Proposition \ref{prop:description-of-measures-by-covers}).
We also describe
 how the system of circuits of each graph is related to ergodic measures.
One of our main results is 
 Theorem \ref{thm:an-ergodic-measure-is-from-circuits},
 which describes a traditional property
 of ergodicity that ``the time average and the space average coincide.''
Neverthless, this is refered in \cite[Theorem 3.3]{BKMS} in the context of Bratteli diagram of finite rank.
It is also known that unique ergodicity is related to uniform convergence
 (see, for example, Theorem 6.19 of \cite{Walters}).
A similar result is seen in Corollary \ref{cor:condition-unique-ergodicity}, the 
complete form of that, in the context of Bratteli diagram of finite rank, is seen in \cite[Theorem 4.11]{BKMS}.
We also extend \cite[Proposition 3.3 (b)]{GM} to get
 Theorem \ref{thm:limited-combinatorics-unique-ergodicity}
 for certain cases of unbounded combinatorics.
The elementary Lemma \ref{lem:positive-matrix} enables us to manage
a sequence of covers with unbounded ranks.
Unlike the graphs considered by Gambaudo and Martens \cite{GM},
 our construction gives rise to the case of a linearly dependent system of circuits.
Section \ref{sec:linear-dependence} is devoted to the study of
 the linear independence of circuits
 (see Theorems \ref{thm:edge-linear} or \ref{thm:linear-independence}).
%
%
%

\vspace{5mm}
\noindent {\sc Acknowledgments.}
The author wishes to express his gratitude to the referee for his kind advices that
 enabled the author to keep in touch with some of works concerning Bratteli--Vershik systems.
%
%
%
%
%
%
%
\section{Preliminaries}\label{sec:preliminaries}
%
%
Let $\Z$ denote the set of all integers, $\Nonne$ be the set of all non-negative integers, and $\Posint$ be the set of all positive integers.
Let $(X,d)$ be a compact metric space,
and $f : X \to X$ be a continuous map with $f(X)=X$.
A pair $(X,f)$ is called a {\it topological dynamical system}.
We mainly consider the case in which $X$ is 0-dimensional.
When $X$ is 0-dimensional, we call $(X,f)$ a {\it 0-dimensional system}.
Let $C$ be the Cantor set.
Topological spaces that are homeomorphic to $C$ are characterized as compact 0-dimensional perfect metrizable spaces.
A topological dynamical system $(X,f)$ is said to be a {\it Cantor system}\/ if $X$ is homeomorphic to $C$.
For $a > b ~(a,b \bi)$, we denote by $[a,b]$ the segment $\seb a, a+1,\dotsc,b \sen$.
A pair $(V,E)$ consisting of a finite set $V$ endowed with the discrete topology and a relation $E \subseteq V \times V$ on $V$ can be considered as a directed graph with vertices $V$ and an edge from $u$ to $v$ when $(u,v) \in E$.
For every $e = (u,v)  \in E$, $u$ is called the {\it initial} vertex, and $v$ is the {\it terminal} vertex. 
The projection from $e = (u,v) \in E$ to the initial (resp. terminal) vertex $u$ (resp. $v$) is denoted by $i$ (resp. $t$).
We call $E$ a {\it surjective relation}\/ when both $i$ and $t$ are surjective.
Thus, for every vertex $v \in V$, there exist edges $(u_1,v),(v,u_2) \in E$.
%
%
In this paper, we call $(V,E)$ a {\it graph}\/ only if $V$ is finite and $E$ is a surjective relation.
%
%
%
%
\begin{defn}
Let $G = (V,E)$ be a graph and $l \bpi$.
A sequence $\pstrz{v}{l}$ of elements of $V$ is a {\it walk} of length $l$ if $(v_i,v_{i+1}) \in E$ for all $0 \le i < l$.
The vertex $v_0$ is called the {\it initial} vertex, and is denoted as $i(p)$;
$v_ l$ is called the {\it terminal} vertex, denoted by $t(p)$.
A walk $\pstrz{v}{l}$ is a {\it path} if $i \ne j$ implies $v_i \neq v_j$,
and $\pstrz{v}{l}$ is a {\it cycle} if $v_0 = v_l$.
A cycle $\pstrz{v}{l}$ is a {\it circuit} if $v_i = v_j$ $(0 \le i < j \le l)$ only when $i = 0$ and $j = l$.
If $p$ is a path of length $l$, then it is also considered to be a graph with $l+1$ vertices and $l$ edges.
If $c$ is a circuit of length $l$, then it is also considered to be a {\it circuit graph} with $l$ vertices and $l$ edges, and $l$ is called the {\it period} of the circuit, denoted as $\Per(c)$.
\end{defn}
%
%
%
\begin{nota}
Let $w = \pstrz{v}{l}$ and $w' = \pstrz{v'}{l'}$ be two walks such that the terminal vertex of $w$ is the initial vertex of $w'$.
Then, we denote the connected walk by $w+w'$.
Namely, $w+w' = (v_ 0,v_ 1, \dotsc,v_ l,{v'}_1,{v'}_2,\dotsc,{v'}_{l'})$.
\end{nota}
\begin{nota}
For a graph $G$, we denote the set of circuit graphs contained in $G$ as:
\[ \sC(G) : = \seb c ~|~ c \text{ is a circuit graph of } G \sen.\]
\end{nota}

\begin{nota}
Let $l \bpi$ and $v \in V$.
We use the following notation:
\[W_v(G,l) := \seb w = (v_0 = v,v_1,v_2,\dotsc,v_l) ~|~ w \text{ is a walk of } G \text{ (of length } l) \sen,\]
\[W(G,l) := \bigcup_ {v \in V}W_v(G,l),\quad W_v(G) := \bigcup_ {l \bpi} W_v(G,l), \text{ and}\]
\[W(G) := \bigcup_ {v \in V, l \bpi} W_v(G,l).\]
\end{nota}
\begin{nota}
Let $l \bpi$.
For $w = (v_0,v_1,v_2,\dotsc,v_l) \in W(G,l)$, we denote:
\[ V(w) := \seb v_0,v_1,v_2,\dots, v_l \sen \myand E(w) := \seb (v_i,v_{i+1}) ~|~ 0 \le i < l \sen.\]
For a circuit graph $c \in \sC(G)$, we similarly denote:
\[ V(c) := \text{ the set of all vertices of } c \myand E(c) := \text{ the set of all edges of } c. \]
\end{nota}

%
%
%
%
%
%
%
%
%
%
\begin{nota}\label{nota:decomp}
Let $X$ be a compact metrizable 0-dimensional space.
The partition of $X$ by non-empty closed and open subsets is called a {\it decomposition}.
The set of all decompositions of $X$ is denoted by $\Decomp(X)$.
Each $\Ucal \in \Decomp(X)$ is endowed with the discrete topology.
\end{nota}
For $\Ucal \in \Decomp(X)$ such that $\Ucal \ne \seb X \sen$, we define:
\[ \dist(\Ucal) := \min \seb d(x,y) ~|~x \in u \in \Ucal,\ y \in v \in \Ucal, \myand u \cap v = \kuu \sen.\]
When $\Ucal = \seb X \sen$, we define $\dist(\Ucal) = \infty$.
Note that for $\Ucal \in \Decomp(X)$, $\dist(\Ucal) = \sup \seb \epsilon\ (\infty \ge \epsilon > 0)~|~d(x,y) < \epsilon \text{ implies that } x \myand y \text{ are in the same element of } \Ucal \sen$.

\begin{nota}\label{nota:vee}
Let $\Ucal,\Vcal \in \Decomp(X)$.
Then, we define $\Ucal \vee \Vcal := \seb U \cap V ~|~ U \in \Ucal, V \in \Vcal, \myand U \cap V \nekuu \sen \in \Decomp(X)$.
We also define $f^{-1}\Ucal := \seb f^{-1}(U) ~|~ U \in \Ucal \sen \in \Decomp(X)$.
\end{nota}
We denote a carrier map as $\kappa_{\Ucal} : X \to \Ucal$, i.e., $x \in \kappa_{\Ucal}(x)$ for all $x \in X$.
Furthermore, we get $U = {\kappa_{\Ucal}}^{-1}(U)$ for all $U \in \Ucal$.
For an arbitrary covering $\Acal$ of $X$, we use the following notation:
\[\mesh(\Acal) := \sup \{\diam(A) ~|~ A \in \Acal \}.\]
%
%
\begin{nota}
If $f : X \to X$ is a surjective mapping and $\Ucal$ is a finite covering of non-empty sets, then a surjective relation $\funcdecomp{f}{\Ucal}$ on $\Ucal$ is defined as:
\[ \funcdecomp{f}{\Ucal} := \seb (u,v) ~|~ f(u) \cap v \nekuu \sen. \]
If the surjective map $f : X \to X$ is regarded as a subset $f \subseteq X \times X$ and $\Ucal$ is a decomposition of $X$, then $\funcdecomp{f}{\Ucal} = (\kappa_{\Ucal} \times \kappa_{\Ucal})(f)$.
In general, $(\Ucal,\funcdecomp{f}{\Ucal})$ is a graph, because $f$ is a surjective relation.
\end{nota}
\begin{defn}\label{defn:cont-edge-surjection}
Let $(X,f)$ be a 0-dimensional system and $(V,E)$ be a graph.
In this paper, a continuous map $\psi : X \to V$ is said to be a {\it continuous}\/ homomorphism from $(X,f)$ to $(V,E)$ if $(\psi \times \psi)(f) \subseteq E$ when $f$ is regarded as a subset $f \subseteq X \times X$.
If a continuous homomorphism $\psi : (X,f) \to (V,E)$ satisfies $(\psi \times \psi)(f) = E$, then we say that $\psi$ is a {\it continuous edge-surjection,}\/ and we write $\psi(X,f)=(V,E)$.
\end{defn}
%
%
Let $G=(V,E)$ and $G'=(V',E')$ be graphs.
A mapping $\fai : V \to V'$ is said to be a graph homomorphism if $(u,v) \in E$ implies $(\fai(u),\fai(v)) \in E'$.
In this case, we write $\fai : G \to G'$ or $\fai : (V,E) \to (V',E')$.
%
%
A map $\fai : V \to V'$ is a graph homomorphism $\fai : (V,E) \to (V',E')$ if and only if $(\fai \times \fai)(E) \subseteq E'$.
We say that $\fai$ is {\it edge-surjective}\/ if $(\fai \times \fai)(E) = E'$.
%
%
%
%
Suppose that a graph homomorphism $\fai : (V_1,E_1) \to (V_2,E_2)$ satisfies the following condition:
\[(u,v),(u,v') \in E_1 \text{ implies that } \fai(v) = \fai(v').\]
Then, in this paper, $\fai$ is said to be {\it \pdirectional}.
If $E_2$ is a map, then every graph homomorphism $\fai$ is \pdirectional.
%
%
\begin{defn}\label{defn:cover}
A graph homomorphism $\fai : G_1 \to G_2$ is called a {\it cover}\/ if it is a \pdirectional\ \tesgh.
\end{defn}
Let $G_1 \getsby{\fai_1} G_2 \getsby{\fai_2} \dotsb$ be a sequence of graph homomorphisms.
For $m > n$, let $\fai_{m,n} := \fai_{n} \circ \fai_{n+1} \circ \dotsb \circ \fai_{m-1}$.
Then, $\fai_{m,n}$ is a graph homomorphism.
If every ${\fai_i}$ $(i \bpi)$ is edge-surjective, then every $\fai_{m,n}$ is edge-surjective,
and if every ${\fai_i}$ $(i \bpi)$ is a cover, then every $\fai_{m,n}$ is a cover.
%
%
To express a 0-dimensional system as the inverse limit of a sequence of graph covers, we cite some lemmas and a theorem from \cite{Shimomura4}.
If $\Gcal$ is a sequence $G_1 \getsby{\fai_1} G_2 \getsby{\fai_2} \dotsb$ of covers, then we always attach the singleton graph $G_0 = (\{0\},\{(0,0)\})$ to the head.
Thus, $\Gcal$ is the sequence $G_0 \getsby{\fai_0} G_1 \getsby{\fai_1} G_2 \getsby{\fai_2} \dotsb$.
%
%
Let us write $G_i = (V_i,E_i)$ for $i \bni$.
Define
\[V_{\Gcal} := \seb \pstrzinf{x} \in \prod_{i = 0}^{\infty}V_i~|~x_i = \fai_i(x_{i+1}) \text{ for all } i \bni \sen \text{ and}\]
\[E_{\Gcal} := \seb (x,y) \in V_{\Gcal} \times V_{\Gcal}~|~(x_i,y_i) \in E_i \text{ for all } i \bni\sen,\]
each equipped with the product topology.
Let $\Gcal$ be a sequence $G_0 \getsby{\fai_0} G_1 \getsby{\fai_1} G_2 \getsby{\fai_2} \dotsb$ of covers.
Then, $V_{\Gcal}$ is a compact, metrizable 0-dimensional space, and the relation $E_{\Gcal}$ determines a continuous  mapping from $V_{\Gcal}$ onto itself (Lemma 3.5 of \cite{Shimomura4}).
%
%
\begin{defn}
Let $\Gcal$ be a sequence $G_0 \getsby{\fai_0} G_1 \getsby{\fai_1} G_2 \getsby{\fai_2} \dotsb$ of covers.
The 0-dimensional system $(V_{\Gcal},E_{\Gcal})$ is called the inverse limit of $\Gcal$, denoted by $G_{\infty}$.

\end{defn}
%
%
For each $i \bni$, the projection from $V_{\Gcal}$ to $V_i$ is denoted by $\fai_{\infty,i}$.
We define:
\[\Ucal_{\Gcal,i} := \seb \fai_{\infty,i}^{-1}(u)~|~u \in V_i \sen,\]
which we can identify with $V_i$ itself.
%
%
Let $i \bni$.
Then, we get $\Ucal_{\Gcal,i} \in \Decomp(V_{\Gcal})$.
Furthermore, by the natural bijection $u \mapsto \fai_{\infty,i}^{-1}(u)~ (u \in V_i)$,
$(\Ucal_{\Gcal,i},\funcdecomp{E_{\Gcal}}{\Ucal_{\Gcal,i}})$ is graph isomorphic to $G_i$.
Fixing a metric on $V_{\Gcal}$,
we get $\mesh(\Ucal_i) \to 0$ as $i \to \infty$.
%
%
%
%
\begin{thm}[{\cite[Theorem 3.9]{Shimomura4}}]\label{thm:structure}
A topological dynamical system is 0-dimensional if and only if it is topologically conjugate to $G_{\infty}$ for some sequence of covers $G_0 \getsby{\fai_0} G_1 \getsby{\fai_1} G_2 \getsby{\fai_2} \dotsb$.
\end{thm}
Instead of a countable sequence of covers, we may use an at most countable inverse system (see the discussion after Theorem 3.9 in \cite{Shimomura4}).
In this paper, we use a sequence of covers.
Let $\Gcal$ be a sequence $G_0 \getsby{\fai_0} G_1 \getsby{\fai_1} G_2 \getsby{\fai_2} \cdots$ of covers.
We define some notation and elementary properties of both $\Gcal$ and the inverse limit $G_{\infty}$ for later sections.
%
%
\itemb
\item[(N-1)]  We write $G_{\infty} = (X,f)$,
\item[(N-2)]  we fix a metric on $X$,
\item[(N-3)]  for each $i \bni$, we write $G_i = (V_i,E_i)$,
\item[(N-4)]  for each $i \bni$, we define $U(v) := \fai^{-1}_{\infty,i}(v)$ for $v \in V_i$ and $\Ucal_{\Gcal,i} := \seb U(v) ~|~v \in V_i\sen \in \Decomp(X)$, and 
\item[(N-5)]  for each $i \bni$, there exists a bijective map $V_i \ni v \leftrightarrow U(v) \in \Ucal_{\Gcal,i}$.
From this bijection, we obtain a graph isomorphism $G_i \cong (\Ucal_{\Gcal,i},\funcdecomp{f}{\Ucal_{\Gcal,i}})$.
\itemn
%
%
It follows that $\mesh(\Ucal_{\Gcal,i}) \to 0$ as $i \to 0$.
%
%

\begin{nota}
Let $l \bpi$ and $x \in X$.
We denote $w(x,l) = (x, f(x), f^2(x),\dotsc,f^l(x))$.
We also denote:
\[W(G_{\infty},l) := \seb w(x,l)~|~ x \in X \sen,\quad W_x(G_{\infty}) := \seb w(x,l)~|~l \bpi \sen, \text{ and}\]
\[W(G_{\infty}) := \seb w(x,l)~|~ x \in X, l \bpi \sen.\]
\end{nota}
For every $\infty \ge m > n$, a map $\fai_{m,n} : W(G_m) \to W(G_n)$ is naturally defined.\begin{nota}
Let $\infty \ge m > n$, $l \bpi$, and $w = \pstrz{v}{l} \in W(G_m,l)$.
We use the following abbreviation:
\[V_n(w) := V(\fai_{m,n}(w)) = \fai_{m,n}(V(w)) \subseteq V_n \quad\text{ and }\quad E_n(w) := E(\fai_{m,n}(w)) = \fai_{m,n}(E(w)) \subseteq E_n.\]
In the same way, for a circuit graph $c \in \sC(G_m)$, we abbreviate as:
\[V_n(c) := \fai_{m,n}(V(c)) \subseteq V_n \quad\text{ and }\quad E_n(c) := \fai_{m,n}(E(c)) \subseteq E_n.\]
\end{nota}

\begin{defn}
Let $m > n\ (m,n \bni)$ and $v \in V_m$.
We say that a walk $w \in W(G_n)$ is {\it $(m,v)$-projective} if there exists a walk $w' \in W(G_m)$ starting from $v$ such that $w = \fai_{m,n}(w')$.
A walk $w \in W(G_n)$ is said to be {\it $m$-projective} if there exists a vertex $v \in V_m$ such that $w$ is $(m,v)$-projective.
Let $l \bpi$ and $x \in X$.
A walk $w  \in W(G_n,l)$ is {\it $(\infty,x)$-projective} if $w =\fai_{\infty,n}(w(x,l))$.
A walk $w \in W(G_n,l)$ is {\it $\infty$-projective} if $w$ is $(\infty,x)$-projective for some $x \in X$.
\end{defn}
Clearly, for $\infty \ge m' > m > n$, an $m'$-projective walk in $W(G_n)$ is $m$-projective.
A walk is $\infty$-projective if and only if a real orbit is projected onto it.
\vspace{0.7cm}
\begin{lem}\label{lem:only(m,v)-projective->real}
Let $m > n\ (m,n \bni)$, $l \bpi$, $v \in V_m$, and $w \in W(G_n,l)$.
Suppose that $\seb w \sen = \fai_{m,n}(W_v(G_m,l))$, i.e., $w$ is the only $(m,v)$-projective walk of length $l$.
Then, $w$ is $\infty$-projective.
\end{lem}
\begin{proof}
Take some $x \in U(v)$.
There exists an $(\infty,x)$-projective walk $w' = (v_0 = v,v_1,\dots,v_l) \in W(G_m,l)$.
Then, $\fai_{m,n}(w')$ is an $(\infty,x)$-projective walk that is also $(m,v)$-projective.
Therefore, we get $w = \fai_{m,n}(w')$, and $w$ is $(\infty,x)$-projective.
\end{proof}

\begin{lem}\label{lem:n+lprojective->only}
Let $n \bni$ and $l \bpi$.
Then, for all $m \ge n+l$ and $v \in V_m$, $\fai_{m,n}(W_v(G_{m},l))$ consists of a single element.
\end{lem}
\begin{proof}
It is sufficient to show that, for all $v \in V_{n+l}$, $\fai_{n+l,n}(W_v(G_{n+l},l))$ consists of a single element.
We show this by induction on $l$.
When $l = 1$, this lemma simply states that $\fai_{n+1,n}$ is a cover.
Suppose that this lemma is satisfied for $l =k-1$.
Let $v \in G_{n+k}$.
Let $w = (v_0 = \fai_{n+k,n+k-1}(v),v_1,v_2,\dots,v_k)  \in \fai_{n+k,n+k-1}(W_v(G_{n+k},k))$.
Because $\fai_{n+k,n+k-1}$ is a cover, $v_1$ is determined independently of the choice of $w$.
We denote this vertex as $u \in V_{n+k-1}$.
We then obtain $\seb (v_1 = u,v_2,\dots,v_k) ~|~ (v_0 = \fai_{n+k,n+k-1}(v),v_1,v_2,\dots,v_k) \in \fai_{n+k,n+k-1}(W_v(G_{n+k},k)) \sen \subseteq W_u(G_{n+k-1},k-1)$.
Then, by the inductive hypothesis, $\fai_{n+k-1,n}(W_u(G_{n+k-1},k-1))$ consists of a single element.
This concludes the proof.
\end{proof}

\begin{prop}\label{prop:m-projective->real}
Let $n \bni$, $l \bpi$ and $m \ge n+l$.
Then, every $w \in \fai_{m,n}(W(G_m,l))$ is $\infty$-projective.
\end{prop}
\begin{proof}
By assumption, there exists some $v \in V_m$ such that $w \in \fai_{m,n}(W_v(G_m,l))$.
By Lemma \ref{lem:n+lprojective->only}, $w$ is the only $(m,v)$-projective walk of length $l$.
Thus, by Lemma \ref{lem:only(m,v)-projective->real}, $w$ is $\infty$-projective.
\end{proof}
%
%
%
%
%
%
%

\section{Minimal Cantor systems}\label{sec:minimalCantorsystems}
A topological dynamical system $(X,f)$ is said to be minimal if there exist no closed subsets $\emptyset  \ne Y \subsetneq X$ that satisfy $f(Y)=Y$.
\begin{prop}
A minimal 0-dimensional system $(X,f)$ is a Cantor system or consists of a single periodic orbit.
\end{prop}
\begin{proof}
Suppose that there exists an isolated point $x_0$.
Because $(X,f)$ is minimal, the $\omega$-limit set $\omega(x_0)$ is $X$ itself.
Because $x_0$ is isolated, $f^n(x_0) = x_0$ for some $n > 0$.
By minimality, $(X,f)$ consists of a single periodic orbit.
It is well known that a compact, metrizable 0-dimensional space is homeomorphic to a Cantor set if there exist no isolated points.
Therefore, if $X$ does not have isolated points, it is homeomorphic to $C$.
\end{proof}

\begin{defn}
Let $\epsilon > 0$.
A (finite) subset $A \subseteq X$ is {\it $\epsilon$-dense} if, for all $y \in X$, there exists some $x \in A$ such that $d(x,y) < \epsilon$.
\end{defn}
\begin{defn}
Let $(X,f)$ be a topological dynamical system and $l \bni$.
An $x \in X$ is {\it $(l,\epsilon)$-dense} if $\seb f^i(x)~|~ 0 \le i \le l \sen$ is $\epsilon$-dense.
\end{defn}

\begin{lem}\label{lem:epsilon-dense}
A topological dynamical system is minimal if and only if the following condition is satisfied.\\
\noindent {\rm (M-1)}\quad For all $\epsilon > 0$, there exists an $l \bni$ such that every $x \in X$ is $(l,\epsilon)$-dense.
\end{lem}
\begin{proof}
Suppose that $(X,f)$ satisfies condition (M-1).
Then, it is clear that every orbit of $(X,f)$ is dense.
Therefore, $(X,f)$ is minimal.
Conversely, suppose that $(X,f)$ is minimal.
Then, every orbit of $(X,f)$ is dense.
Let $\epsilon > 0$ and $x \in X$.
Then, there exists an $l_x \bni$ such that $x$ is $(l_x,\epsilon)$-dense.
As this is an open condition, there exists a neighborhood $U$ of $x$ such that, for all $x' \in U$, $x'$ is $(l_x,\epsilon)$-dense.
Because $X$ is compact, by a standard argument, we can find an $l \bni$ satisfying condition (M-1).
\end{proof}
Let $\Gcal$ be a sequence $G_0 \getsby{\fai_0} G_1 \getsby{\fai_1} G_2 \getsby{\fai_2} \dotsb$ of covers.
We use the notation of \usualnotationss given in \S \ref{sec:preliminaries}.
\begin{thm}\label{thm:minimal-main}
Let $(X,f)$ be a Cantor system.
Then, the following are mutually equivalent:
\begin{enumerate}
\item\label{thm:minimal-main:minimal} $(X,f)$ is minimal,
%
%
\item\label{thm:minimal-main:Edge} $ \forall n \bni,\ \exists N > n \myand \exists l \bpi, \text{ such that } \forall m \ge N \myand \forall w \in W(G_m,l),\ E_n(w) = E_n$,
\item\label{thm:minimal-main:Vertex} $ \forall n \bni,\ \exists N > n \myand \exists l \bpi, \text{ such that } \forall m \ge N \myand \forall w \in W(G_m,l),\ V_n(w) = V_n$,
%
%
\item\label{thm:minimal-main:Cycle:Edge} $ \forall n \bni,\ \exists N > n \text{ such that } \forall m \ge N \myand \forall c \in \sC(G_m),\ E_n(c) = E_n$,
\item\label{thm:minimal-main:Cycle:Vertex} $ \forall n \bni,\ \exists N > n \text{ such that } \forall m \ge N \myand \forall c \in \sC(G_m),\ V_n(c) = V_n$,
%
%
\item\label{thm:minimal-main:Edge2} $ \forall n \bni,\ \exists N > n, \forall m \ge N, \exists l_m \bpi \mysuchthat \forall w \in W(G_m,l_m),\ E_n(w) = E_n$, and
\item\label{thm:minimal-main:Vertex2} $ \forall n \bni,\ \exists N > n, \forall m \ge N, \exists l_m \bpi \mysuchthat \forall w \in W(G_m,l_m),\ V_n(w) = V_n$.
\end{enumerate}
\end{thm}
\begin{proof}
%
%
It is clear that \ref{thm:minimal-main:Edge} implies \ref{thm:minimal-main:Vertex},
\ref{thm:minimal-main:Cycle:Edge} implies \ref{thm:minimal-main:Cycle:Vertex}, and
\ref{thm:minimal-main:Edge2} implies \ref{thm:minimal-main:Vertex2}.
%
%
Let us show that \ref{thm:minimal-main:minimal} implies \ref{thm:minimal-main:Edge}.
Suppose that $(X,f)$ is minimal.
Let $n \bni$.
Let $\Ucal = \Ucal_{\Gcal,n} \vee f^{-1}\Ucal_{\Gcal,n}$, which consists of non-empty sets by definition.
Let $u,u' \in \Ucal_{\Gcal,n}$.
Then, we have that $u \cap f^{-1}(u') \nekuu$ if and only if $(u,u') \in \funcdecomp{f}{\Ucal_{\Gcal,n}}$.
The map $V_n \ni u \mapsto U(u) \in \Ucal_{\Gcal,n}$ is a graph isomorphism from $(V_n,E_n)$ to $(\Ucal_{\Gcal,n},\funcdecomp{f}{\Ucal_{\Gcal,n}})$.
For every $x \in X$ and $k \bpi$, there exists a $v_i \in \Ucal~(0 \le i \le k)$ such that $f^i(x) \in v_i$ for every $0 \le i \le k$.
If we write $v_i = u_i \cap f^{-1}(u'_{i}) \in \Ucal$ for $0 \le i \le k$, then every edge $(u_i,u'_i)$ $(0 \le i \le k)$ is an element of $E_n(w(x,k+1))$.
Let $0 < \epsilon \le \dist(\Ucal)$.
By Lemma \ref{lem:epsilon-dense}, there exists an $l' \bpi$ such that all $x \in X$ are $(l',\epsilon)$-dense.
Thus, for every $x \in X$, $w(x,l')$ passes through all elements of $\Ucal$.
Let $l = l'+1$.
Then, $E_n(w(x,l)) = E_n$ for all $x \in X$.
Because $\mesh(\Ucal_{\Gcal,j}) \to 0$ as $j \to \infty$, there exists an $N' > n$ such that $\mesh(\Ucal_{\Gcal,N'}) < \epsilon$.
Hence, $\Ucal_{\Gcal,N'}$ refines $\Ucal$.
Let $N = N'+l$ and $m \ge N$,
and take some arbitrary $w \in W(G_m,l)$.
By Proposition \ref{prop:m-projective->real}, $\fai_{m,N'}(w)$ is $\infty$-projective.
Therefore, there exists an $x \in X$ such that $\fai_{m,N'}(w) = \fai_{\infty,N'}(w(x,l))$.
Hence, $E_n(w) = E_n(\fai_{m,N'}(w)) = E_n(\fai_{\infty,N'}(w(x,l))) = E_n(w(x,l)) = E_n$, as desired.
%
%
Because treading on a circuit is a special kind of walk, it is evident that \ref{thm:minimal-main:Edge} implies \ref{thm:minimal-main:Cycle:Edge}.
In the same way, \ref{thm:minimal-main:Vertex} implies \ref{thm:minimal-main:Cycle:Vertex}.
%
%
Let us show that \ref{thm:minimal-main:Cycle:Edge} implies \ref{thm:minimal-main:Edge2}.
Suppose that \ref{thm:minimal-main:Cycle:Edge} is satisfied,
and let $n \bni$.
Then, there exists an $N > n$ such that, for all $m \ge N$ and for all $c \in \sC(G_m)$, $E_n(c) = E_n$.
If we fix some $m \ge N$ and $l_m \ge \sharp V_m$,
then for every $w = \pstrz{v}{l_m} \in W(G_m,l_m)$,  there exist $s,t \in [0,l_m]$ ($s < t$) with $v_s = v_t$.
Thus, there exists a $c \in \sC(G_m)$ such that $w$ passes through all edges of $c$.
By the assumption, it follows that $E_n(c) = E_n$.
Therefore, as desired, we have $E_n(w) \supseteq E_n(c) = E_n$.
In the same way, \ref{thm:minimal-main:Cycle:Vertex} implies \ref{thm:minimal-main:Vertex2}.
%
%
Let us show that \ref{thm:minimal-main:Vertex2} implies \ref{thm:minimal-main:minimal}.
Assume that condition \ref{thm:minimal-main:Vertex2} is satisfied,
and let $\epsilon > 0$.
Take $n \bni$ such that $\mesh(\Ucal_{\Gcal,n}) < \epsilon$.
Then, by the assumption, there exist  $l \bpi$ and $m > n$ such that, for all $w \in W(G_m,l)$, it follows that $V_n(w) = V_n$.
Let $x \in X$.
Because $\fai_{\infty,m}(w(x,l)) \in W(G_m,l)$, it then follows that $V_n(w(x,l)) = V_n$.
Therefore, $x$ is $(l,\epsilon)$-dense.
By Lemma \ref{lem:epsilon-dense}, we have the desired result.
\end{proof}

\section{Invariant measures}\label{sec:invariant-measures}
In the first half of \S 3 of \cite{GM}, Gambaudo and Martens describe the combinatorial construction of invariant measures in relation to graph covers.
Following their work, we examine this for our case of 0-dimensional systems and general graph covers.
We are forced to use the notion of the graph circuit explicitly.
The system of graph circuits will represent all invariant measures.
In the second half of this section, we investigate the conditions under which our graph circuit representations are able to represent ergodic measures.
Consequently, we can formalize the condition for a 0-dimensional system to be uniquely ergodic by means of graph covers.
Let $\Gcal$ be a sequence $G_0 \getsby{\fai_0} G_1 \getsby{\fai_1} G_2 \getsby{\fai_2} \dotsb$ of covers.
We use the notation of \usualnotationss given in \S \ref{sec:preliminaries}.
Let us denote the set of all invariant Borel measures on $(X,f)$ as $\sM_f = \sM_f(X)$, and the invariant Borel probability measures as $\sP_f = \sP_f(X)$.
Both $\sM_f$ and $\sP_f$ are endowed with the weak topology.
For $n \bni$ and $v \in V_n$,
we denote:
\[ E_{n,-}(v) := \seb (u,v)~|~(u,v) \in E_n \sen \myand E_{n,+}(v) := \seb (v,u)~|~(v,u) \in E_n \sen. \]
For each $e = (u,v) \in E_n$, we denote a closed and open set $U(e) := U(u) \cap f^{-1}(U(v))$.
Note that for $e \neq e'~(e,e' \in E_n)$, it follows that $U(e) \cap U(e') = \kuu$.
\begin{lem}\label{lem:inv-on-vertex}
Let $n \bni$ and $\mu \in \sM_f$.
Then, for every $v \in V_n$, $f^{-1}(U(v))$ (resp. $U(v)$) can be written as $\bigcup_ {e \in E_{n,-}(v)} U(e)$ (resp. $\bigcup_ {e \in E_{n,+}(v)} U(e)$).
In particular, it follows that
\begin{equation}\label{eqn:measure-invariant}
\sum_{e \in E_{n,-}(v)}\mu(U(e)) = \sum_{e \in E_{n,+}(v)}\mu(U(e)).
\end{equation}
\end{lem}
\begin{proof}
It is easy to see that $f^{-1}(U(v)) = \bigcup_ {e \in E_{n,-}(v)}(U(e))$ and $U(v) = \bigcup_ {e \in E_{n,+}(v)}(U(e))$, and these are both disjoint unions.
Therefore, we have the first statement.
Because $\mu$ is an invariant measure, the next statement follows.
\end{proof}
We denote $\sO_V := \seb U(v) ~|~ v \in V_n, n \bni \sen = \bigcup_ {n \bni} \Ucal_{\Gcal,n}$ and $\sO_E := \seb U(e)~|~ e \in E_n, n \bni \sen$.
Both $\sO_V$ and $\sO_E$ are open bases of the topology of $X$.
For an arbitrary Borel measure $\mu$ on $X$, it follows that
\begin{equation}\label{eqn:measure-vertex}
\text{for all } m > n \text{ and for all } v \in V_n, \sum_{u \in V_m, \fai_{m,n}(u) = v}\mu(U(u)) = \mu(U(v)),
\end{equation}
\begin{equation}\label{eqn:measure-edge}
\text{for all } m > n \text{ and for all } e \in E_n, \sum_{e' \in E_m, \fai_{m,n}(e') = e}\mu(U(e')) = \mu(U(e)).
\end{equation}
Conversely, suppose that non-negative real values $\mu(U)$ are assigned for all $U \in \sO_V$ (resp. $U \in \sO_E$) to satisfy (\ref{eqn:measure-vertex}) (resp. (\ref{eqn:measure-edge})).
Then, $\mu$ is a countably additive measure on the set of all open sets of $X$.
Therefore, as is well known, $\mu$ can be extended to the unique Borel measure on $X$.
A Borel measure $\mu$ on $X$, constructed in this way, is a probability measure if and only if $1 = \sum_{v \in V_n} \mu(v)$ for all $n \bni$.

\begin{lem}\label{lem:creation-of-measures}
Suppose that non-negative real values $\mu(U)$ are assigned for all $U \in \sO_E$ such that both (\ref{eqn:measure-invariant}) and (\ref{eqn:measure-edge}) are satisfied.
Then, the resulting $\mu$ is an invariant measure.
\end{lem}
\begin{proof}
Because (\ref{eqn:measure-edge}) is satisfied, $\mu$ is a Borel measure.
Let $U(v) \in \sO_V$ with $v \in V_n$ and $n \bni$.
As has been observed, we obtain $U(v) = \bigcup_ {e \in E_{n,+}(v)}(U(e))$ and $f^{-1}(U(v)) = \bigcup_ {e \in E_{n,-}(v)}(U(e))$, both of which are disjoint unions.
Therefore, by (\ref{eqn:measure-invariant}), we have $\mu(f^{-1}(U(v))) = \mu(U(v))$.
Because $v \in V_n$ and $n \bni$ are arbitrary, a standard argument shows that $\mu$ is an invariant measure.
\end{proof}
For a graph $G = (V,E)$, we denote $\Mcal(G) := \seb \sum_{e \in E} a(e)e~|~a(e) \ge 0 \myforall e \in E \sen$.
For each $v \in V$, we define the following equation:
\begin{equation*}\label{eqn:module-invariant}
\Eq(v): \qquad \sum_{e \in E_{n,-}(v)}a(e) = \sum_{e \in E_{n,+}(v)}a(e).
\end{equation*}
We denote $\Ical(G) := \seb \sum_{e \in E} a(e)e \in \Mcal(G)~|~ \Eq(v) \text{ is satisfied for all } v \in V \sen \subseteq \Mcal(G)$, and $\Pcal(G) := \seb \sum_{e \in E} a(e)e \in \Ical(G) ~|~ \sum_{e \in E}a(e) = 1 \sen \subset \Ical(G)$.
\begin{lem}\label{lem:graph-hom-linear}
Let $\fai : G' = (V',E') \to G = (V,E)$ be a cover and $\sum_{e \in E'} a(e)e \in \Ical(G')$.
Then, a linear map $\fai_* : \Ical(G') \to \Ical(G)$ is defined by $\fai_*(\sum_{e \in E'} a(e)e) := \sum_{e \in E'}a(e)\fai(e)$.
Furthermore, it follows that $\fai_*(\Pcal(G')) \subseteq \Pcal(G)$.
\end{lem}
\begin{proof}
Let $x \in \Ical(G')$.
We have to show that $\fai_*(x) \in \Ical(G)$.
For each $v' \in V'$, equation $\Eq(v')$ is satisfied.
Let $v \in V$.
The sum of the equations in $\Eq(v')$ ($v' \in \fai^{-1}(v)$) implies equation $\Eq(v)$ for $\fai_*(x)$.
Therefore, it follows that $\fai_*(x) \in \Ical(G)$, as desired.
Thus, it is clear that $\fai_*(\Pcal(G')) \subseteq \Pcal(G)$.
\end{proof}

\begin{nota}
Let $G = (V,E)$ be a graph, $c \in \sC(G)$, and $s \in \Real$.
Then, we denote $sc := \sum_{e \in E(c)}se$.
For $c \in \sC(G)$, we consider $c = 1 \cdot c \in \Mcal(G)$,
and write $\tilc = (1/\Per(c)) \cdot c$.
For a subset $\Ccal \subseteq \sC(G)$, we can define:
\[ \Hull(\Ccal) := \seb \sum_{c \in \Ccal} s(c)\tilc~|~s(c) \ge 0 \myforall c \in \Ccal,~\sum_{c \in \Ccal}s(c) = 1 \sen.\]

\end{nota}
The following lemma is obvious, and we therefore omit a proof.
\begin{lem}
Let $G = (V,E)$ be a graph and $c \in \sC(G)$.
It follows that $c \in \Ical(G)$ and $\tilc \in \Pcal(G)$.
\end{lem}
As the next proposition shows, $\Ical(G)$ is spanned by $\sC(G)$.
\begin{prop}\label{prop:non-negative-coef-for-circuits}
Let $G = (V,E)$ be a graph and $x \in \Ical(G)$.
Then, there exists $\seb s(c) \ge 0 \sen_{c \in \sC(G)}$ such that $x = \sum_{c \in \sC(G)}s(c)c$.
\end{prop}
\begin{proof}
Let $x = \sum_{e \in E}x(e)e$, where $x(e) \ge 0$ for all $e \in E$.
Let $n := \sharp \seb e ~|~ x(e) > 0 \sen$.
We use an induction on $n$.
Suppose $n = 1$.
Let $e_0 \in E$ be such that $x(e_0) > 0$, and let $v = t(e_0)$.
Because $n = 1$, $e_0$ is the only edge for which $x(e_0) > 0$,
and because $\Eq(v)$ is satisfied by every $v \in V$, it follows that $t(e_0) = i(e_0)$.
Hence, a singleton graph $(\seb v \sen, \seb e_0 = (v,v) \sen)$ that is also a circuit is contained in $G$.
Call this circuit $c_0$.
Then, $x = s(c_0)c_0$, as desired.
We now assume that the claim is satisfied for $n = k$,
and consider the case in which $n = k+1$.
Let $x \in \Ical(G)$ be written as $x = \sum_{e \in E} x(e)\cdot e$ with $x(e) \ge 0$ for all $e \in E$ and $\sharp \seb e ~|~ x(e) > 0 \sen = k+1$.
Then, there exists an edge $e_0 \in E$ such that $x(e_0)$ is the smallest positive value, i.e., $x(e) \ge x(e_0)$ or $x(e) = 0$ for all $e \in E$.
We claim that we can select successive edges $e_i~(i = 0,1,2, \dotsc)$ such that, if $e_i$ is determined, then $e_{i+1} \in E$ satisfies $t(e_i) = i(e_{i+1})$, and $x(e_i) \ge x(e_0)$ for all $i = 1,2,\dotsc$.
To show this, suppose that $e_i$ has been selected.
Because $E$ is a surjective relation, there exists a non-empty set of edges starting from $t(e_i)$. 
By $\Eq(t(e_i))$, there exists an edge $e$ in this set with $x(e) > 0$.
Once we have $x(e) > 0$, we also have that $x(e) \ge x(e_0)$ from the minimality of $x(e_0)$.
Because $\sharp E < \infty$, we have a circuit $c$ connecting some successive edges $e_l,e_{l+1},\dotsc,e_{l'}$.
Let $\alpha := \min \seb x(e_i) ~|~ l \le i \le l' \sen$ and $x(e_{i_0}) = \alpha$ with some $l \le i_0 \le l'$.
Let $x' = x - \alpha \cdot c$.
Then, the coefficient of $e_{i_0}$ in the expression of $x'$ is $0$.
Thus, if we write $x' = \sum_{e \in E}x'(e) \cdot e$,
then $\sharp \seb e ~|~ x'(e) > 0 \sen \le k$, and the similar equations are satisfied for all vertices.
By the inductive hypothesis, we have the desired result.
\end{proof}
An obvious consequence is the following, for which we omit a proof.
\begin{prop}
$\Pcal(G) = \Hull(\sC(G))$.
\end{prop}
\begin{rem}\label{rem:coef-for-independent-circuits}
For linearly independent $\sC(G)$, the expression $x = \sum_{c \in \sC(G)}x(c)c$ is unique for every $x \in \Ical(G)$.
Therefore, in this expression, $x(c) \ge 0$ for all $c \in \sC(G)$.
\end{rem}
Note that, in general, $\sC(G) \subset \Ical(G)$ is not linearly independent.
Therefore, even if $x \in \Ical(G)$ is expressed as $x = \sum_{c \in \sC(G)}s(c)c$, there may be a $c \in \sC(G)$ with $s(c) < 0$.
By Proposition \ref{prop:non-negative-coef-for-circuits}, we can take a basis as a subset of $\sC(G)$.
Because $\fai_{m,n} :G_m \to G_n$ is a graph cover for each $m > n$, there exists a linear map $(\fai_{m,n})_* : \Ical(G_m) \to \Ical(G_n)$.
As has been shown in Lemma \ref{lem:graph-hom-linear}, it follows that $(\fai_{m,n})_*(\Ical(G_m)) \subseteq \Ical(G_n)$, and $(\fai_{m,n})_*(\Pcal(G_m)) \subseteq \Pcal(G_n)$.
\begin{nota}\label{nota:basis}
For each $n \bni$, we denote the dimension of $\Ical(G_n)$ as $d_n$, and $\sC_n := \sC(G_n)$.
We fix a basis $B_n = \seb c_{n,1}, c_{n,2}, \dotsc,c_{n,d_n} \sen \subseteq \sC_n$ for $\Ical(G_n)$.
\end{nota}

\begin{nota}\label{nota:deltan}
We write $\Delta_n := \Pcal(G_n)$ and, for each $m > n$, define an affine map as $\xi_{m,n} := (\fai_{m,n})_*|_{\Delta_m} : \Delta_m \to \Delta_n$.
If $c \in \sC_n$, then $\tilc \in \Delta_n$.
\end{nota}
\begin{nota}
We denote an inverse limit as follows:
\begin{align*} \Icalinf := \seb (x_0,x_1,\dotsc) ~|~&\myforall i \bni,~ x_i \in \Ical(G_i), \\
 & \hspace{5mm}\text{and for all } m > n,~ (\fai_{m,n})_*(x_m) = x_n \sen,
\end{align*}
\[\Deltainf := \seb (x_0,x_1,\dotsc) ~|~\myforall i \bni,~ x_i \in \Delta_i,~\text{and for all } m > n,~ \xi_{m,n}(x_m) = x_n \sen.\]
Both $\Icalinf$ and $\Deltainf$ are endowed with the product topology, and these correspond to the subspace topology.
For each $n \bni$, we denote the projection as $\xi_{\infty,n} : \Deltainf \to \Delta_n$.
It is evident that $\xi_{\infty,n}$ is continuous.
For $\infty \ge k > j > i$, it follows that $\xi_{j,i} \circ \xi_{k,j} = \xi_{k,i}$.

\end{nota}
It is well known that $\sP_f$ is a compact, metrizable space.
By definition, $\Deltainf$ is also a compact, metrizable space.
The following lemma is obvious, and so we omit a proof.
\begin{lem}\label{lem:subsequence-converge}
Let $\bN \subseteq \Nonne$ be an infinite subset.
Let $\seb x_n \sen_{n \in \bN}$ be a sequence such that $x_n \in \Delta_n$ for all $n \in \bN$.
Then, there exists a subsequence $\seb n_i \sen_{i \bpi}$ such that, for each $m \bni$,
the sequence $\seb \xi_{n_i,m}(x_{n_i}) \sen_{n_i > m}$ converges to an element $y_m \in \Delta_m$.
In this case, it follows that $\pstrzinf{y} \in \Deltainf$.
\end{lem}
\begin{lem}\label{lem:from-infinity-to-n}
Let $n \bni$.
Then, the sequence $\Delta_n \supseteq \xi_{n+1,n}(\Delta_{n+1}) \supseteq (\xi_{n+2,n})(\Delta_{n+2}) \supseteq \dotsb$ is non-increasing, and $\bigcap_{m>n}\xi_{m,n}(\Delta_m) = \xi_{\infty,n}(\Deltainf)$.
\end{lem}
\begin{proof}
Because $\xi_{j,i}\circ\xi_{k,j} = \xi_{k,i}$ for $k > j > i$,
the first statement is clear; also, the inclusion $\bigcap_ {m>n}\xi_{m,n}(\Delta_m) \supseteq \xi_{\infty,n}(\Deltainf)$.
We have to show the other inclusion.
Let $n \bni$,
and let $O \subset \Delta_n$ be an open set such that $\xi_{\infty,n}(\Deltainf) \subset O$.
Suppose that there exists an element $x \in \bigcap_ {m>n}\xi_{m,n}(\Delta_m) \cap O^c$.
Then, there exists a sequence $\seb x_m \sen_{m > n}$ such that $x_m \in \Delta_m~(m > n)$ and $\xi_{m,n}(x_m) = x$.
Let $x_n = x$ and $x_i = \xi_{n,i}(x)$ for $0 \le i < n$.
By Lemma \ref{lem:subsequence-converge}, we have a contradiction.
Thus, $\bigcap_ {m>n}\xi_{m,n}(\Delta_m) \subseteq O$.
Because $O$ is arbitrary, we reach the desired conclusion.
\end{proof}
\begin{nota}
Let $\mu \in \sM_f$.
We write $\bar{\mu} := \pstrzinf{\mu}$, where $\mu_n = \sum_{e \in E_n}\mu(U(e))e$ for all $n \bni$.
Then, it follows that $\bar{\mu} \in \Icalinf$.
Furthermore, if $\mu \in \sP_f$, we have that $\bar{\mu} \in \Deltainf$ and $\mu_n \in \Delta_n$ for each $n \bni$.
Following \cite{GM}, we denote the map $\mu \to \bar{\mu}$ as $p_*$.
Thus, we have the maps $p_* : \sM_f \to \Icalinf$ and $p_* : \sP_f \to \Deltainf$.
\end{nota}
We now check Proposition 3.2 of \cite{GM} for our case.
\begin{prop}\label{prop:description-of-measures-by-covers}
The maps $p_* : \sM_f \to \Icalinf$ 
and $p_* : \sP_f \to \Deltainf$ are isomorphisms.
\end{prop}
\begin{proof}
From Lemma \ref{lem:creation-of-measures}, we have the bijectivity of these maps.
Their linearity is clear.
Let us show that $p_*$ is a homeomorphism.
For a closed and open subset $U$ of $X$, let $\delta_U$ be the characteristic function from $X$ to $\Real$.
Let $\mu \in \sM_f$, $n \bni$, and $e \in E_n$.
Then, it follows that $\int \delta_{U(e)} d\mu = \mu(U(e))$.
Thus, it is easy to see that convergence in $\sM_f$ implies convergence in $\Icalinf$.
Conversely, convergence of each $\mu(U(e))$ implies the convergence of $\int \delta_{U(e)} d\mu$.
Thus, by a standard argument, the map is a homeomorphism, as required.
The second statement is also clear.
\end{proof}
In Notation \ref{nota:basis}, we have fixed a basis $B_n = \seb c_{n,1}, c_{n,2}, \dotsc,c_{n,d_n} \sen \subseteq \sC_n$ for $\Ical(G_n)$ for each $n \bni$.
Therefore, we can state the following:
\begin{prop}\label{prop:unique-expression}
Let $\mu \in \sM_f$.
Then, there exists a unique expression $\mu_n = \sum_{1 \le i \le d_n}s_{n,i}\tilc_{n,i}~(s_{n,i} \in \Real,~ 1 \le i \le d_n,~ n \bni)$.
\end{prop}
\begin{nota}
Let $\seb x_{n_i} \sen_{i \bpi}$ be a sequence such that $n_1 < n_2 < \dotsb$, and $x_{n_i} \in \Delta_{n_i}$ for all $i \bpi$.
Suppose that, for all $m \bni$, $\lim_{i \to \infty}\xi_{n_i,m}(x_{n_i}) = y_m \in \Delta_m$.
Then, we denote $\lim_{i \to \infty}x_{n_i} := (y_0,y_1,\dotsc) \in \Deltainf$.
\end{nota}
From hereon, we consider restrictions of the total set $\Nonne$ to infinite subsets of $\Nonne$.
Let $\bN \subseteq \Nonne$ be an infinite subset.
Because the restricted sequence $\seb G_n \sen_{n \in \bN}$ of $\Gcal$ induces the same inverse limit $G_{\infty}$, it is convenient to think of such restrictions.
We believe this will make it easier to manage expressions than to reconstruct the sequence $\Gcal$ itself.
\begin{defn}
Let $\bN \subseteq \Nonne$ be an infinite subset.
Let $\Ccal = \Ccal(\bN)$ be a sequence $\seb \Ccal_n \sen_{n \in \bN}$ of subsets $\Ccal_n \subseteq \sC_{n}$ for all $n \in \bN$.
We say that $\Ccal$ is a {\it system of circuits} enumerated by $\bN$.
We denote as follows:
\[\Ccalinf := \seb \seb c_{n} \sen_{n \in \bN}~|~c_n \in \Ccal_n \myforall n \in \bN, \lim_{n \to \infty}\tilc_{n} \text{ exists} \sen.
\]
For each $x = \seb c_{n} \sen_{n \in \bN} \in \Ccalinf$, we denote $\bar{x} = \lim_{n \to \infty}\tilc_{n} \in \Deltainf$.
We write $\ovCcalinf := \seb \bar{x}~|~x \in \Ccalinf \sen$.
Note that it is evident from Lemma \ref{lem:subsequence-converge} that, for all sequences of circuits $\seb c_{n} \sen_{n \in \bN}$ with $c_n \in \Ccal_n$ for all $n \in \bN$, there exists a subsequence $\seb n_i \sen_{i \bpi}$ of $\bN$ such that there exists a $\lim_{i \to \infty}\tilc_{n_i}$.
Nevertheless, this does not directly imply that $\Ccalinf \nekuu$.
\end{defn}
%
%
%
%
\begin{defn}
Let $\mu \in \sM_f$ be an invariant measure,
and let $\bN$ be an infinite subset of $\Nonne$.
Let $\Ccal$ be a system of circuits enumerated by $\bN$.
We say that $\Ccal$ {\it expresses} $\mu$ if $\bar{\mu}$ is written as $\pstrzinf{\mu}$ with $\mu_n = \sum_{c \in \Ccal_n}s(c)\tilc$\quad $(s(c) \ge 0 \myforall c \in \Ccal_n)$ for all $n \in \bN$.
\end{defn}
As a consequence of this definition, a system of circuits $\Ccal$ expresses all measures if and only if $\xi_{\infty,n}(\Delta_{\infty}) \subseteq \Hull(\Ccal_n)$ for all $n \in \bN$.
We now discuss the relation between the expression by means of $\Ccal$ and the ergodicity.
Supposed that $\mu$ is an ergodic measure and $\Ccal$ is a system of circuits that expresses $\mu$.
Then, $\bar{\mu}$ is written as $\pstrzinf{\mu}$ with $\mu_n = \sum_{c \in \Ccal_n}s(c)\tilc$\quad $(s(c) \ge 0 \myforall c \in \Ccal_n)$ for all $n \in \bN$.
As Theorem \ref{thm:ergodic-measures-are-from-essential-circuits} will show, $\bar{\mu}$ is considered to be a limit of circuits in $\Ccal$.
Some circuits of $\Ccal$ may not contribute to the limit $\bar{\mu}$.
It may be possible to choose circuits from those that have positive weights in the above expression,
but, even if we take sequences of circuits that have positive weights, the resulting measures may be far from $\bar{\mu}$.
Should the sum of weights of circuits that generate measures far from $\bar{\mu}$ be less than that of the circuits that generate measures near $\bar{\mu}$?
The next theorem answers this question in the affirmative.
This seems to be a repetition of the statement that the ergodicity is the coincidence of the time average and the space average.
The increase in $n$ for a graph $G_n$ gives a simultaneous increase in time and refinement of space.
%
%
The next theorem is essentially seen in its exact form in \cite[Theorem 3.3]{BKMS} in the context of Bratteli diagrams of finite rank.
In our way, we think we have to describe rather vaguely as is seen in the next theorem, because we are treating general 0-dimensional systems or general sequences of graph covers.
%
%
\begin{thm}\label{thm:an-ergodic-measure-is-from-circuits}
Let $\mu \in \sP_f$.
Let $\bN \subseteq \Nonne$ be an infinite subset, and
let $\Ccal$ be a system of circuits enumerated by $\bN$.
Suppose that $\mu$ is expressed by $\Ccal$, and 
$\bar{\mu}$ is written as $\pstrzinf{\mu}$ with $\mu_n = \sum_{c \in \Ccal_n}s(c)\tilc$\quad $(s(c) \ge 0 \myforall c \in \Ccal_n)$ for all $n \in \bN$.
For $n > m$ and $\epsilon > 0$, let 
$\Ccal(\mu,m, n,\epsilon) := \seb c \in \Ccal_n~|~d(\mu_m,\xi_{n,m}(\tilc)) \le \epsilon \sen$.
Let us consider the following conditions:
\enumb
\item\label{item:ergodic} $\mu$ is ergodic, and
\item\label{item:sumlim1} for all $m \in \bN$ and all $\epsilon > 0$, $\lim_{n \to \infty}\sum_{c \in \Ccal(\mu,m, n,\epsilon)}s(c) = 1$.
\enumn
Then, \ref{item:ergodic} implies \ref{item:sumlim1}.
Moreover, if each $\sC_n$ (not $\Ccal_n$) is linearly independent for all $n \in \bN$,
then \ref{item:ergodic} and \ref{item:sumlim1} are equivalent.
\end{thm}
\begin{proof}
Let $\mu \in \sP_f$.
Let $\bN$ and $\Ccal$ be as above.
Suppose that $\bar{\mu} = \pstrzinf{\mu} \in \Deltainf$ is written as
$\mu_i = \sum_{c \in \Ccal_i} s(c)\tilc$ with $s(c) \ge 0$ for all $c \in \Ccal_i$ and for all $i \in \bN$.
Let us show that \ref{item:ergodic} implies \ref{item:sumlim1}.
Suppose that $\mu$ is ergodic.
Fix an arbitrary $m$ and $\epsilon > 0$.
We prove by contradiction.
Suppose that there exists an $\epsilon' > 0$ such that there exist infinitely many $n > m~(n \in \bN)$ with $\sum_{c \in \Dcal(n,\epsilon)}s(c) \ge \epsilon'$, where $\Dcal(n,\epsilon) := \seb c \in \Ccal_n~|~d(\mu_m,\xi_{n,m}(\tilc)) > \epsilon \sen$.
Let $\bN'$ be the set of all such $n > m~(n \in \bN)$.
We write the space of unit vectors in the hyperplane extending $\Delta_m$ as $\Bcal_1$.
We define $\Vcal := \seb v \in \Bcal_1~|~\text{there exists a small } \delta > 0 \text{ such that } \mu_m+\delta v \in \Delta_m \sen$.
We need to consider the case in which $\Delta_m$ is not a singleton.
Because $\Delta_m$ is a convex combination in a Euclidean space, it follows that $\Vcal \nekuu$.
Let us fix $0 < \theta < \pi/2$.
For every $v \in \Vcal$, let
\[ \Vcal(v) := \left\{\, x \in \Delta_m~\Big|~x \neq \mu_m \myand \frac{(x - \mu_m)\cdot v}{|x-\mu_m|} \ge \cos\theta \,\right\}.\]
Then, there exists a finite set $\seb v_1,v_2, \dotsc, v_K \sen \subset \Vcal$ such that 
$\Delta_m\setminus \seb \mu_m \sen = \bigcup_ {1 \le k \le K} \Vcal(v_k)$.
For $n \in \bN'$ and $1 \le k \le K$, we define 
\[\Dcal(n,\epsilon,k):= \seb c \in \Dcal(n,\epsilon) ~|~ \xi_{n,m}(\tilc) \in \Vcal(v_k)\sen.\]
It follows that $\Dcal(n,\epsilon) = \bigcup_ {1 \le k \le K} \Dcal(n,\epsilon,k)$.
For $1 \le k \le K$, we define $a(n,k) := \sum_{c \in \Dcal(n,\epsilon,k)} s(c)$.
Then, for all $n \in \bN'$, it follows that $\epsilon' \le \sum_{1 \le k \le K}a(n,k)$.
Therefore, there exist a $k_0\ (1 \le k_0 \le K)$, a subsequence $\seb n_i \sen_{i \bpi}$, and an $a > 0$ such that $a(n_i,k_0) \to a$ as $i \to \infty$.
Because $v_{k_0}$ and $0 < \theta < \pi/2$ is fixed in a certain direction from $\mu_m$, it is evident that $a < 1$.
Let $\mu_1(n_i) := \sum_{c \in \Dcal(n_i,\epsilon,k_0)}\left(s(c)/a(n_i,k_0)\right)\tilc$ and $\mu_2(n_i) := \sum_{c \in \Ccal_{n_i} \setminus \Dcal(n_i,\epsilon,k_0)}\left(s(c)/(1 - a(n_i,k_0)) \right) \tilc$ for sufficiently large $i$.
It follows that $\mu_{n_i} = a(n_i,k_0)\mu_1(n_i)+(1-a(n_i,k_0))\mu_2(n_i)$.
Taking a subsequence if necessary, we obtain the limits $\mu_1 = \lim_{i \to \infty}\mu_1(n_i) \in \Deltainf$ and $\mu_2 = \lim_{i \to \infty}\mu_2(n_i) \in \Deltainf$.
It follows that $\bar{\mu} = a \mu_1 + (1-a)\mu_2$.
This contradicts the hypothesis that $\mu$ is ergodic.
%
%
%
%
%
%
To show the last statement, we assume that each $\sC_n$ $(n \in \bN)$ is linearly independent.
It is enough to show that \ref{item:sumlim1} implies \ref{item:ergodic}.
The expression $\mu_i = \sum_{c \in \Ccal_i} s(c)\tilc$ $(i \in \bN)$ is extended to $\mu_i = \sum_{c \in \sC_i}s(c)\tilc$ $(i \in \bN)$ with $s(c) = 0$ for all $c \in \sC_i\setminus \Ccal_i$ $(i \in \bN)$.
We prove by contradiction.
Suppose that \ref{item:sumlim1} is satisfied and $\mu$ is not ergodic.
Then, there exist $\mu_1, \mu_2 \in \sP_f$ and some $t_0~(0 < t_0 < 1)$ such that $\mu = (1-t_0) \mu_1 + t_0 \mu_2$.
Let us write $\bar{\mu}_{\alpha}~(\alpha = 1,2)$ as $\bar{\mu}_{\alpha} = (\mu_{\alpha,0},\mu_{\alpha,1}, \mu_{\alpha,2}, \mu_{\alpha,3}, \dotsc) \in \Deltainf$, where
$\mu_{\alpha,i} = \sum_{c \in \sC_i} s(\alpha,c)\tilc$ with $s(\alpha,c) \ge 0$ for all $c \in \sC_i$ and for all $i \in \bN$.
The linear independence implies that $s(c) = (1-t_0)s(1,c)+t_0 s(2,c)$ for all $c \in \sC_i$ $(i \in \bN)$.
There exists an $m \in \bN$ such that $\mu_{1,m} \neq \mu_{2,m}$.
Connecting $\mu_{1,m}$ and $\mu_{2,m}$, we get a line $\ell(t) := (1-t)\mu_{1,m}+t\mu_{2,m}$ $(t \in \Real)$.
Then, there exists a point $\mu_{\rm mid} := \ell(t_0/2)$ between $\mu_{1,m}$ and $\mu_m$.
Let $p : \Delta_m \to \ell$ be a linear projection onto $\ell$, and
let $\ell(-\infty,t_0/2)$ be the half line.
Then, $H := p^{-1}(\ell(-\infty,t_0/2))$ is a convex open set that contains $\mu_{1,m}$.
Let $\epsilon > 0$ be small enough that the $\epsilon$-ball around $\mu_m$ has no points in common with $H$.
Take an arbitrary $n > m$, and define $\Dcal(n) := \seb c \in \sC_n~|~\xi_{n,m}(\tilc) \in H \sen$.
Then, it follows that $\Dcal(n) \cap \Ccal(\mu,m, n, \epsilon) = \kuu$.
Because $\xi_{n,m}\left(\sum_{c \in \sC_n}s(1,c)\tilc\right) = \mu_{1,m} \in H$, there exist constants $\epsilon_1 > 0$ and $N > 0$ such that $\sum_{c \in \Dcal(n)}s(1,c) \ge \epsilon_1$ for all $n \ge N$.
As $\epsilon_1 \le \sum_{c \in \Dcal(n)}s(1,c) \le \sum_{c \in \sC_n \setminus \Ccal(\mu,m,n, \epsilon)}s(c)$ for all $n \ge N$,
it follows that $\sum_{c \in \Ccal(\mu,m,n, \epsilon)}s(c) \le 1-\epsilon_1$ for all $n > N$.
This contradicts condition \ref{item:sumlim1}.
\end{proof}
\begin{rem}
In Example \ref{example:sumlim1-but-non-ergodic}, we construct a sequence of graph covers such that $\sC_n$ is not linearly independent for each $n \bni$, in which case there exists a non-ergodic measure such that condition \ref{item:sumlim1} is satisfied.
\end{rem}
\begin{thm}\label{thm:ergodic-measures-are-from-essential-circuits}
Let $\mu \in \sP_f$ be an ergodic measure.
Let $\bN \subseteq \Nonne$ be an infinite subset.
Suppose that $\Ccal$ is a system of circuits that expresses $\mu$.
Then, it follows that $\bar{\mu} \in \ovCcalinf$.
\end{thm}
\begin{proof}
As a partial consequence of Theorem \ref{thm:an-ergodic-measure-is-from-circuits}, for all $m \bni$, there exists an $N(m) > m$ such that, for all $n$ with $N(m) \le n \in \bN$, there is some $c_{n} \in \Ccal_{n}$ such that $d(\mu_i,\xi_{n,i}(\tilc_{n})) \le 1/m$ for all $0 \le i \le m$.
Therefore, we can find $x \in \Ccalinf$ such that $\bar{x} = \bar{\mu}$.
\end{proof}

\begin{cor}\label{cor:ergodic-measures-are-from-circuits}
Let $\bN \subseteq \Nonne$ be an infinite subset.
Let $\Ccal$ be a system of circuits, enumerated by $\bN$, that expresses all ergodic measures.
Then, $\ovCcalinf$ contains all ergodic measures.
In particular, $\Ccalinf \nekuu$.
\end{cor}
As a consequence, we obtain an expression for unique ergodicity by means of a property of the sequences of graph covers.
The next corollary should be an addition to theorems that state the relation between unique ergodicity and  uniform convergence.
\begin{cor}\label{cor:condition-unique-ergodicity}
For a 0-dimensional system $(X,f)$, the following are equivalent:
\enumb
\item\label{item:uniquely-ergodic} $(X,f)$ is uniquely ergodic,
\item\label{item:forall-unique-limit} for all systems of circuits enumerated by infinite sets that express all ergodic measures, $\ovCcalinf$ consists of a single element,
\item\label{item:exists-unique-limit} there exists a system of circuits enumerated by an infinite set that expresses all ergodic measures such that $\ovCcalinf$ consists of a single element, and
\item\label{item:uniform-convergence} for all circuits $c_n \in \sC_n~(n \bni)$, $\xi_{n,m}(\tilc_n)$ converges uniformly to an element of $\Delta_m$.
\enumn
\end{cor}
\begin{proof}
The equivalence of \ref{item:uniquely-ergodic}, \ref{item:forall-unique-limit}, and \ref{item:exists-unique-limit} are obvious from Corollary \ref{cor:ergodic-measures-are-from-circuits}.
Clearly, \ref{item:uniform-convergence} implies \ref{item:forall-unique-limit}.
We show that \ref{item:forall-unique-limit} implies \ref{item:uniform-convergence}.
Let $\seb c_n \sen~(c_n \in \sC_n)$ be a sequence of circuits such that $\xi_{n,m}(\tilc_n)$ converges to some $\mu_m \in \Delta_m$ for all $m \bni$.
Suppose that the convergence is not uniform with respect to the selection of $\seb c_n \sen$.
Then, there exists an $m \bni$ such that we can find $\epsilon > 0$, an arbitrarily large $n > m$, and some $c_n \in \sC_n$ such that $d(\xi_{n,m}(\tilc_n), \mu_m) \ge \epsilon$.
The contradiction is now obvious by Lemma \ref{lem:subsequence-converge}.
\end{proof}
Let $\mu \in \sP_f$ and $\bar{\mu} = \pstrzinf{\mu} \in \Deltainf$.
In Notation \ref{nota:basis}, we have fixed a system of linearly independent bases $B_n = \seb c_{n,1}, c_{n,2}, \dotsc,c_{n,d_n} \sen \subseteq \sC_n$ of $\Ical(G_n)$ for each $n \bni$.
Thus, for each $n \bni$, there exists a unique expression $\mu_n = \sum_{1 \le i \le d_n}s(\mu)_{n,i}\tilc_{n,i}$ such that $\sum_{1 \le i \le d_n}s(\mu)_{n,i} = 1$.
In this context, we obtain the following expression of the above corollary.
\begin{cor}\label{cor:condition-unique-ergodicity-2}
The 0-dimensional system $(X,f)$ is uniquely ergodic if and only if there exists a sequence of real numbers $s_{m,i}$ $(m \bni,\ 1 \le i \le d_m)$ with $\sum_{1 \le i \le d_m}s_{m,i} =1$ such that:
\itemb
\item for every circuit $c_n \in \sC_n~(n \bni)$,
$\xi_{n,m}(\tilc_n)$ converges uniformly to $\sum_{1 \le i \le d_m}s_{m,i}\tilc_{m,i}$ as $n \to \infty$.
\itemn
\end{cor}
%
%
%
%
\section{Linear dependence of circuits}\label{sec:linear-dependence}
%
%
%
In this section, we  discuss the conditions for the linear dependence of $\sC(G)$ for a graph $G$.
Let $G = (V,E)$ be a graph and $c_0 \in \sC(G)$.
A linear combination of circuits $\sum_{c \in \sC(G),~c \ne c_0}s(c)c$ is expressed simply as $\sum_{c \ne c_0}s(c)c$.
Suppose that there exists a linear combination
\[\sum_{1 \le i \le d} s(c_i)c_i = 0 \quad (s(c_i) \in \Real \myforall 1 \le i \le d). \]
Let $e$ be an edge of $c_i$ with $s(c_i) \ne 0$.
Then, there exists a $j \ne i$ such that $c_j$ contains $e$.
We show that the converse also holds.
\begin{thm}\label{thm:edge-linear}
Let $G = (V,E)$ be a graph and $c_0 \in \sC(G)$.
There exists a linear combination of circuits
\[c_0 = \sum_{c \ne c_0}s(c)c \quad (s(c) \in \Real \myforall c \ne c_0)\]
if and only if, for every edge $e$ of $c_0$, there exists a $c \in \sC(G)\setminus \seb c_0 \sen$ that contains $e$.
In this case, the coefficients $s(c)~(c \ne c_0)$ can be taken from the set of rational numbers.
\end{thm}
\begin{proof}
From the remark above, if the equation in Theorem~\ref{thm:edge-linear} holds, then each edge of $c_0$ is contained in a circuit $c \ne c_0$.
We have to show the converse.
Suppose that every edge of $c_0$ is contained by another circuit.
Let $\sC(c_0)$ be the set of all circuits that have common edges with $c_0$, and
let $c \in \sC(c_0)$.\\
Notation: Let $P$ be a subset of the set of all paths of $G$, and $p \in P$.
We say $p$ is {\it maximal} in $P$ if there exists no path $q \in P$ such that $p$ is a proper sub-path of $q$.\\
Let $P(c_0,c) := \seb p ~|~p \text{ is a maximal common path of } c_0 \myand c \sen$, 
 $P(c_0) := \bigcup_{c \in \sC(c_0)}P(c_0,c)$,
and $\Pmax(c_0) := \seb p \in P(c_0)~|~p \text{ is maximal in } P(c_0) \sen$.\\
\noindent Notation: For a circuit $c$ and vertices $a \ne b$ on it, the path on $c$ from $a$ to $b$ is denoted by $p(c;a,b)$.
For mutually distinct vertices $x_1,x_2,\dotsc,x_k$ of $c_0$, we use the expression $x_1 \le  x_2 \le \dotsb \le x_k$ in the following sense:
\itemb
\item $p(c_0;x_i,x_{i+2})$ contains $x_{i+1}$ for all $(1 \le i \le n-2)$, i.e., we check the order for every three steps.
\itemn
In the above notation, if $x \le y$ and $x \ne y$, then we write $x < y$.
We claim that, for every $2 \le n$, there exists a sequence of paths $\seb p_1,p_2,\dotsc,p_n \sen \subseteq \Pmax(c_0)$ with $p_i = p(c_0;a_i,b_i)$ for $1 \le i \le n$ such that the following is satisfied:
\itemb
\item $a_i < a_{i+1} \le b_i < b_{i+1}$ for all $1 \le i < n$.
\itemn
To show this by induction, take an arbitrary $p_1 \in \Pmax(c_0)$.
Assume that $p_{k}$ is determined for $1 \le k$, and
let $e$ be the edge of $c_0$ with $i(e) = b_k$.
We take some $p_{k+1} \in \Pmax(c_0)$ that contains $e$.
Because $p_k$ is maximal, it follows that $a_k < a_{k+1} \le b_k < b_{k+1}$.
The induction is complete.
There exists the minimal $n \ge 2$ such that $p(c_0;a_n,b_n)$ contains $a_1$.
We fix such an $n \ge 2$.
Note that $a_n < a_1 \le b_n$.
Because $p(c_0;a_1,b_1)$ is maximal, it follows that $a_1 \le b_n < b_1$.
For each $1 \le i \le n$, there exists a circuit $c_i$ such that one of the maximal common paths with $c_0$ is $p_i$.
In the next calculation of multiple walks, if we encounter an expression $p(c;v,v)$ for a circuit $c$ and a vertex $v$, then we just delete it from the list.

\begin{equation*}
\begin{array}{llll}
\sum_{1 \le i \le n}c_i &=& \quad p(c_0;a_1,b_n) +p(c_0;b_n,b_1) + p(c_1;b_1,a_1) &\ \text{decomposing } c_1\\
 & &\   +p(c_0;a_2,b_1) +p(c_0;b_1,b_2) + p(c_2;b_2,a_2) &\ \text{decomposing } c_2\\
 & &\   \qquad \vdots \\
 & &\   +p(c_0;a_n,b_{n-1}) +p(c_0;b_{n-1},b_n) + p(c_n;b_n,a_n) & \ \text{decomposing } c_n \\
 & & \\
 & =&   \quad p(c_0;a_1,b_n) & \\ 
 & &\   +p(c_0;a_n,b_{n-1}) +p(c_0;b_{n-1},b_n) + p(c_n;b_n,a_n) & \ \text{decomposing } c_n\\
 & &\   +p(c_0;a_{n-1},b_{n-2}) +p(c_0;b_{n-2},b_{n-1}) + p(c_{n-1};b_{n-1},a_{n-1}) & \ \text{decomposing } c_{n-1}\\
 & &\    \qquad \vdots \\
 & &\   +p(c_0;a_2,b_1) +p(c_0;b_1,b_2) + p(c_2;b_2,a_2) &\ \text{decomposing } c_2\\
 & &\   +p(c_0;b_n,b_1) + p(c_1;b_1,a_1) & \\ 
\end{array}
\end{equation*}
\begin{equation*}
\begin{array}{lllll}
\phantom{\sum_{1 \le i \le n}c_i}& =&   \quad p(c_0;a_1,b_n) & & \\ 
 & &\   +p(c_n;b_n,a_n) +p(c_0;a_n,b_{n-1})  & +p(c_0;b_{n-1},b_n)  & \ \text{decomposing } c_n\\
 & &\   + p(c_{n-1};b_{n-1},a_{n-1}) + p(c_0;a_{n-1},b_{n-2}) & +p(c_0;b_{n-2},b_{n-1})  & \ \text{decomposing } c_{n-1}\\
 & &\    \qquad \vdots & \qquad \vdots\\
 & &\   +p(c_2;b_2,a_2)+p(c_0;a_2,b_1) & +p(c_0;b_1,b_2)   &\ \text{decomposing } c_2\\
 & &\   +p(c_1;b_1,a_1) & +p(c_0;b_n,b_1)   & \ 
\end{array}
\end{equation*}
In the final expression, the consecutive treading of paths in the left-hand side forms a walk that begins and ends at $a_1$.
Therefore, it is a cycle.
On the other hand, the totality of the right-hand side equals $c_0$.
Therefore, we have that $\sum_{1 \le i \le n}c_i = (\text{a cycle}) + c_0$.
Generally, a cycle can be decomposed into a sum of circuits.
Thus, the coefficient of $c_0$ on the right-hand side is positive.
This concludes the proof.
\end{proof}
The following is a direct consequence of Theorem \ref{thm:edge-linear}.
\begin{thm}\label{thm:linear-independence}
Let $G = (V,E)$ be a graph.
Then, $\sC(G)$ is linearly independent if and only if, for every circuit $c \in \sC(G)$, there exists an edge $e$ of $c$ such that no other circuits have $e$ as an edge.
\end{thm}

\begin{lem}\label{lem:circuit-isnot-positive-linear-combination1}
Let $G = (V,E)$ be a graph and $c_0 \in \sC(G)$.
Suppose that there exists a linear combination $c_0 = \sum_{c \neq c_0}s(c)c$.
Then, there exists a $c \neq c_0$ such that $s(c) < 0$.
\end{lem}
\begin{proof}
Suppose that the equation $c_0 = \sum_{c \neq c_0}s(c)c$ is satisfied.
Then, there exists a circuit $c \neq c_0$ with $s(c) > 0$ that has a common edge with $c_0$.
It follows that $c$ leaves $c_0$ at some vertex $v$ with an edge $e'$ of $c$, i.e., $i(e') = v$ and $e'$ is not an edge of $c_0$.
Because we have to subtract some circuit that has an edge $e'$ to obtain $c_0$, we reach the conclusion.
\end{proof}

\begin{lem}\label{lem:notallpositive}
Let $G = (V,E)$ be a graph and $c_0 \in \sC(G)$.
Then, there does not exist any linear combination $c_0 = \sum_{c \neq c_0}s(c)c$ such that $s(c) \ge 0$ for all $c \neq c_0$.
\end{lem}
\begin{proof}
This is obvious from Lemma \ref{lem:circuit-isnot-positive-linear-combination1}.
\end{proof}
\begin{prop}\label{prop:every-circuit-is-extremal}
Let $G = (V,E)$ be a graph.
Then, for every $c \in \sC(G)$, $\tilc$ is an extremal point of $\Pcal(G)$.
\end{prop}
\begin{proof}
This is obvious from Lemma \ref{lem:notallpositive}.
\end{proof}

\begin{example}\label{example:circuits-tree}
In the case of graphs considered by Gambaudo and Martens \cite{GM}, the sets of all circuits are linearly independent.
We examine another class of directed graphs for which the sets of all circuits are again linearly independent.
Nevertheless, we do not know whether such graphs have covers that generate all minimal Cantor systems.
Roughly speaking, we consider graphs whose circuits form trees.
Let $G$ be a directed graph and $\sC$ be the set of all circuits of $G$.
We assume that if $c$ and $c'$ are distinct elements of $\sC$, then $\sharp(V(c) \cap V(c')) \le 1$.
We say a vertex $v \in V(c) \cap V(c')$ is a connecting vertex if it exists.
We consider an undirected graph such that the set of vertices is $\sC$ and the set of edges is $\seb \seb c,c'\sen~|~c \ne c' \myand \sharp(V(c) \cap V(c')) = 1 \sen$.
We assume that this undirected graph is connected and has no cycle, i.e., it is a tree.
Then, there exists a cycle in $G$ that passes all the edges of $G$ exactly once and all the connecting vertices exactly twice.
Let $\seb s(c) \sen_{c \in \sC}$ be an arbitrary set of positive integers.
Then, for a circuit graph $C$ with a suitable period, there exists a homomorphism $\fai : C \to G$ such that $\fai_*(C) = \sum_{c \in \sC}s(c)c$.
Even if an initial edge in $C$ is fixed, the map $\fai$ can be taken such that the initial edge is mapped to an arbitrary edge of $G$.
Let $G'$ be another directed graph of this type, and $\sC'$ be the set of all circuits of $G'$.
Let us choose the system of positive integers $\seb m_{c',c} \sen_{c' \in \sC, c \in \sC}$ arbitrarily.
If we modify the path lengths of all circuits of $G'$, we can construct a map $\fai : G' \to G$ such that $\fai_*(c') = \sum_{c \in \sC}m_{c',c}c$ for all $c' \in \sC'$.
Note that even if a circuit $c'$ of $G'$ winds round a circuit $c$ of $G$ only once, $\fai^{-1}(c)$ may consist of fragmented paths.
Let $\seb d_n \sen_{n \bni}$ be an arbitrary sequence of positive integers and $M_n~(n \bni)$ be an arbitrary system of $(d_{n+1},d_n)$-matrices such that each entry is a  positive integer.
By the above argument, we can find a sequence of graph covers of this type that generate this system of `winding' matrices with positive integer entries.
\end{example}
%
%
%
%
\section{Finite ergodicity and bounded combinatorics}\label{sec:finite-and-bounded}
In this section, we ensure that the argument of Gambaudo and Martens \cite[Proposition 3.3]{GM} is still valid in our case, with an extension to the somewhat unbounded case for \cite[Proposition 3.3 (b)]{GM}.
Let $\Gcal$ be a sequence $G_0 \getsby{\fai_0} G_1 \getsby{\fai_1} G_2 \getsby{\fai_2} \dotsb$ of covers.
We use the notation of \usualnotationss given in \S \ref{sec:preliminaries}, particularly $(X,f) = G_{\infty}$.

\begin{prop}\label{prop:countably-distinct-ergodic-measures-from-circuits}
Let $\bN$ be an infinite subset of $\Nonne$.
Let $\Ccal$ be a system of circuits enumerated by $\bN$ that generates all measures.
Suppose that $\sP_f$ has at least $k~(1 \le k \le \infty)$ ergodic measures.
Then, there exist $k$ sequences $\seb c(n,j)~|~ n \in \bN,~c(n,j) \in \Ccal_{n} \sen$ with $0 \le j < k$ such that $x_j = \lim_{n \to \infty}\tilc(n,j) ~(0 \le j < k)$ are mutually distinct ergodic measures.
\end{prop}
\begin{proof}
Let $\mu_0$ be an ergodic measure.
By Corollary \ref{cor:ergodic-measures-are-from-circuits}, there exists a sequence $\seb c(n,0) \sen_{n \in \bN}$ such that $\bar{\mu}_0 = \lim_{n \to \infty} \tilc(n,0)$.
Suppose that there exists another ergodic measure $\mu_1$.
By Corollary \ref{cor:ergodic-measures-are-from-circuits}, there exists a sequence $\seb c(n,1) \sen_{n \in \bN} \in \Ccalinf$ such that $\bar{\mu}_1 = \lim_{n \to \infty} \tilc(n,1)$.
In this way, we can construct $c(n,j)$ for $0 \le j < k$ and $n \in \bN$.
\end{proof}

We want to check \cite[Proposition 3.3 (a)]{GM} for our case.

\begin{thm}\label{thm:finite-ergodicity}
Let $\bN$ be an infinite subset of $\Nonne$.
Let $\Ccal$ be a system of circuits that expresses all measures enumerated by $\bN$ and $k \bpi$.
Suppose that $\sharp \Ccal_n \le k$ for all $n \in \bN$, and that there exist mutually distinct circuits $c'(n,i)~(1 \le i \le l)$ in $\Ccal_n$ such that $\tilc'(n,i)$ converges to a non-ergodic measure for each $1 \le i \le l$.
Then, $\sP_f$ has at most $k - l$ ergodic measures.
\end{thm}
\begin{proof}
Suppose that there exist $k-l+1$ ergodic measures $\mu_j$ $(1 \le j \le k-l+1)$.
Then, by Proposition \ref{prop:countably-distinct-ergodic-measures-from-circuits}, there exist sequences $\seb c(n,j) \sen_{n \in \bN}$ with $c(n,j) \in \Ccal_n$ for each $1 \le j \le k-l+1$ such that $\bar{\mu}_j = \lim_{n \to \infty} \tilc(n,j)$ for each $1 \le j \le k+1$.
Let $n \in \bN$.
Because $\seb c(n,j)~|~1 \le j \le k-l+1 \sen \cup \seb c'(n,i)~|~1 \le i \le l \sen \subseteq \Ccal_n$, one of the following cases occurs:
\enumb
\item there exists a pair $j < j'$ such that $c(n,j) = c(n,j')$, or
\item there exists a pair $j$ and $i$ such that $c(n,j) = c'(n,i)$.
\enumn
Obviously, the number of such combinations is bounded.
Therefore, there exists a pair $j_0 < j'_0$ such that $c(n,j_0) = c(n,j'_0)$ for an infinite number of $n$, or there exists a pair $j_0$ and $i_0$ such that $c(n,j_0) = c'(n,i_0)$ for an infinite number of $n$.
Thus, we have a contradiction.
\end{proof}
\begin{example}\label{example:sumlim1-but-non-ergodic}
We construct a sequence of covers $G_0 \getsby{\fai_0} G_1 \getsby{\fai_1} G_2 \getsby{\fai_2} \dotsb$ that satisfies all of the following:
\itemb
\item $(X,f)$ is minimal,
\item  each $\sC_n~(n \bpi)$ is linearly dependent,
\item  $(X,f)$ has exactly two ergodic measures, and
\item  there exists a non-ergodic measure $\mu$ and an expression of $\mu$ by $\sC_n$ $(n \bni)$ such that \ref{item:sumlim1} of Theorem \ref{thm:an-ergodic-measure-is-from-circuits} is satisfied.
\itemn
Let $n \bpi$.
Let $v_{n,1},v_{n,2},v_{n,3}$ be distinct vertices.
Construct two paths $a_{n,1}$ and $b_{n,1}$ from $v_{n,1}$ to $v_{n,2}$,
and two paths $a_{n,2}$ and $b_{n,2}$ from $v_{n,2}$ to $v_{n,3}$.
We assume that the lengths of $a_{n,1}$, $b_{n,1}$, $a_{n,2}$, and $b_{n,2}$ are the same.
Construct a path $d_n$ from $v_{n,3}$ to $v_{n,1}$.
We assume that the paths above have no common vertices, except for the initial  and  terminal vertices.
Thus, $\sC_{n}$ consists of four circuits, namely $\seb a_n = a_{n,1} + a_{n,2} + d_{n}, b_n = b_{n,1} + b_{n,2} + d_{n}, c_n = a_{n,1} + b_{n,2} + d_{n}, c'_n = b_{n,1} + a_{n,2} + d_{n} \sen$.
Because the linear equation $a_n + b_n = c_n + c'_n$ exists, $\sC$ is linearly dependent.
For a cover $\fai_{n+1,n} : G_{n+1} \to G_n$, we assume that $\fai(v_{n+1,i}) = v_{n,1}$ $(i = 1,2,3)$ and all paths $a_{n+1,j},b_{n+1,j}$ $(j = 1,2)$ and $d_{n+1}$ are mapped to circuits $\seb a_n, b_n, c_n, c'_n \sen$ with the following multiplicity.
Let $p(n) \to +\infty$ sufficiently quickly as $n \to +\infty$.
The path $a_{n+1,1}$ is mapped to $a_n$ $p(n)$-times and to $b_n$ once.
The first edge of $a_{n+1,1}$ is mapped to the direction of $a_n$.
The first edge of the path $b_{n+1,1}$ has to be mapped to the same direction as $a_{n+1,1}$.
We assume that $b_{n+1,1}$ is mapped to $a_n$ once and to $b_n$  $p(n)$-times.
The path $a_{n+1,2}$ is mapped to $a_n$ $p(n)$-times and to $b_n$ once.
The path $b_{n+1,2}$ is mapped to $a_n$ once and to $b_n$ $p(n)$-times.
The path $d_{n+1}$ is mapped to $c_n$ once and to $c'_n$ once.
Then, because the condition \ref{thm:minimal-main:Cycle:Edge} in Theorem \ref{thm:minimal-main} is satisfied, the inverse limit $(X,f)$ is minimal.
All $\Delta_n$ $(n \bpi)$ are convex quadrilaterals with extremal points $\tila_n,\tilb_n,\tilc_n$, and $\tilc'_n$.
It follows that $\xi_n$ maps $\tila_{n+1}$ and $\tilb_{n+1}$ closer to $\tila_n$ and $\tilb_n$, respectively, as $n \to \infty$.
On the other hand, $\xi_n$ maps $\tilc_{n+1}$ and $\tilc'_{n+1}$ closer to $(\tila_n +\tilb_n)/2$ as $n \to \infty$.
Because we have assumed that $p(n) \to +\infty$ sufficiently quickly, $\tila_{\infty} = \lim_{n \to \infty}\tila_n$ and $\tilb_{\infty} = \lim_{n \to \infty}\tilb_n$ exist, and are distinct ergodic measures.
Moreover, we have that $\lim_{n \to \infty}\tilc_n = \lim_{n \to \infty}\tilc'_n = (\tila_{\infty}+\tilb_{\infty})/2$.
By Theorem \ref{thm:finite-ergodicity}, there exist exactly two ergodic measures.
Let $\mu = (\tila_{\infty} + \tilb_{\infty})/2$.
This is not ergodic, and has an expression $\mu_n = (1/2)\tila_n + (1/2)\tilb_n = (1/2)\tilc_n + (1/2)\tilc'_n$ for all $n \bpi$.
Let $\epsilon > 0$.
Then, we have seen that, for sufficiently large $n$, $c_n, c'_n \in \Ccal(\mu,m,n,\epsilon)$.
Thus, condition \ref{item:sumlim1} of Theorem \ref{thm:an-ergodic-measure-is-from-circuits} is satisfied for a non-ergodic measure $\mu$.
Note that the resulting $\sP_f$ is one-dimensional, and requires only two circuits $a_n$ and $b_n$ to express all measures for $n \bni$.
If $p(n) \to +\infty$ too slowly, the system becomes uniquely ergodic.
\end{example}
In \cite[Proposition 3.3 (b)]{GM}, Gambaudo and Martens showed that if $f$ has bounded combinatorics for a Cantor minimal system, then it is uniquely ergodic.
We want to extend this theorem in our case, and also loosen the condition to the somewhat unbounded case.
Let $\Gcal$ be a sequence $G_0 \getsby{\fai_0} G_1 \getsby{\fai_1} G_2 \getsby{\fai_2} \cdots$ of covers for a 0-dimensional system.
We use the notation of \usualnotationss given in \S \ref{sec:preliminaries}.
Let $n \bni$.
For a circuit $c \in \sC(G_{n+1})$, there exists a representation $(\fai_n)_*(c) = \sum_{c' \in \sC_n}s(c')c'$ with $s(c') \bni$ for all $c' \in \sC_n$.
Generally, these representations are not unique.
To make matters worse, some of the $s(c')$ may be equal to $0$, meaning that the argument in \cite[Proposition 3.3 (b)]{GM} fails in our case.
Nevertheless, we think it is not unnatural to restrict the systems of circuits to those in which all values of $s(c')$ are positive.
This is partly because, as shown in Example \ref{example:sumlim1-but-non-ergodic}, there exist cases in which not all circuits are needed to express all measures, and partly because, as shown in Example \ref{example:circuits-tree}, there exist cases in which all $\sC_n$ $(n \bni)$ are linearly independent and all values of $s(c')$ become positive for certain covers of minimal Cantor systems.
As in the case of \cite[Proposition 3.3 (b)]{GM}, we consider a ``winding matrix'' for all $\fai_n$ $(n \bni)$.
Let $\Ccal$ be a system of circuits enumerated by $\Nonne$ that expresses all invariant measures.
Suppose that there exist non-negative integers $m(c,c')~(c \in \Ccal_{n+1}, ~c' \in \Ccal_n,~n \bni)$ such that $(\fai_n)_*(c) = \sum_{c' \in \Ccal_n}m(c,c')c'$ for all $n \bni$.
We fix a numbering $\seb c(n,1),c(n,2),\dotsc,c(n,d_n) \sen = \Ccal_n$ for all $n \bni$.
For each $n \bni$, we define a matrix $M_n$ whose entries are $m_n(i,j) = m(c(n+1,i),c(n,j))$ for all $1 \le i \le d_{n+1}$ and $1 \le j \le d_n$.
Let $n \bni$.
Let $\alpha \in \Delta_{n+1}$ be expressed as $\alpha = \sum_{c \in \Ccal_{n+1}}s(c)\tilc$ with $\sum_{c \in \Ccal}s(c) = 1$.
Then, $(\fai_n)_*(\alpha) = \sum_{c \in \Ccal_{n+1}} \sum_{c' \in \Ccal_n}s(c)\cdot(\Per(c')/\Per(c)) \cdot m(c,c')) \cdot \tilc'$.
If we write $\alpha_i := s(c(n+1,i))$ for $1 \le i \le d_{n+1}$ and $l_{s}(i) := \Per(c(s,i))$ for $1 \le i \le d_s$ for all $s \bni$, 
then $(\fai_n)_*(\alpha) = \sum_{i=1}^{d_{n+1}}\sum_{j = 1}^{d_n}\alpha_i \cdot (l_{n}(j)/l_{n+1}(i)) \cdot m_n(i,j) \cdot \tilc(n,j)$.
Therefore, we consider a matrix $\barM_n$ whose $(i,j)$-th entry is $\barm_n(i,j) = (l_{n}(j)/l_{n+1}(i)) \cdot m_n(i,j)$.
The matrix $\barM_n$ acts on $\alpha$ from the right.
We say $\seb \barM_n \sen_{n \bni}$ is the {\it system of winding matrices} of $\Ccal$.
Note that $\sum_{1 \le j \le d_n}\barm_n(i,j) =1$ for all $1 \le i \le d_{n+1}$.
\begin{example}\label{example:arbitrary-rational-entries}
Let $\seb d_n \sen_{n \bni}$ be a sequence of positive integers.
Let $\Mcal : \seb \barM_n = (\barm_n(i,j)) \sen_{n \bni}$ be a sequence of $(d_{n+1},d_n)$-matrices with positive, rational entries such that $\sum_j \barm_n(i,j) = 1$.
As shown in Example \ref{example:circuits-tree}, for any sequence $m_n(i,j)$ $(n \bni, 1 \le i \le d_{n+1},  1 \le j \le d_n)$ of positive integers, there exists a sequence of covers $\Gcal : G_0 \getsby{\fai_0} G_1 \getsby{\fai_1} \dotsb$ that generates such a sequence.
Note that we are only concerned with the rate $m_n(i,1):m_n(i,2):\dotsb:m_n(i,j)$ to obtain the desired $\barM_n$.
Therefore, we can construct $\Gcal$ such that, for each $n \bni$, all circuits in $G_n$ have an equal period,
and we can adjust the values of $m_n(i,j)$ to generate the system of $\barm_n(i,j)$.
Thus, we can realize $\Mcal$ by a sequence of covers.
\end{example}

\begin{nota}
Let $M$ be a matrix whose entries are $m_{i,j}$, and
let $\epsilon$ be a real number.
We write $\epsilon < M$ if $\epsilon < m_{i,j}$ for all $i,j$.
We call $M$ a positive matrix if $0 < M$.
\end{nota}

\begin{lem}\label{lem:positive-matrix}
Let $s,t$ be positive integers and $1/t \ge \epsilon > 0$.
Let $M = (m_{i,j})_{1 \le i \le s,~1 \le j \le t}$ be a matrix that satisfies $\epsilon \le M$ and $\sum_{1 \le j \le t}m_{i,j} = 1$ for all $1 \le i \le n$.
Both $\Real^s$ and $\Real^t$ are endowed with the norm as $\abs{x} = \sum_i \abs{x_i}$, where $x = \pstro{x}{s}$ or $\pstro{x}{t}$.
Let $0 \ne x = \pstro{x}{s} \in \Real^s$ be orthogonal, in the usual sense, to $(1,1,\dotsc,1) \in \Real^s$, i.e., $\sum_{1 \le i \le s}x_i = 0$.
Let $y = xM$.
Then, it follows that $\abs{y} \le (1-t\epsilon)\abs{x}$.
\end{lem}
\begin{proof}
Although the calculation is elementary, we list it here for completeness.
We set $I = [1,s]$ and $J = [1,t]$.
First, $\sum_{j \in J}y_j = \sum_{i,j}m_{i,j}x_i = \sum_{i} \left(\sum_j m_{i,j}\right)x_i = \sum_i x_i = 0$, and $y$ is also orthogonal to $(1,1,\dotsc,1)$.
Let $I_1 := \seb i \in I ~|~ x_i \ge 0 \sen$ and $I_2 := [1,s]\setminus I_1$
and $J_1 := \seb j \in J ~|~ y_j \ge 0 \sen$ and $J_2 := [1,t]\setminus J_1$.
Then, $\abs{x} = \sum_{i \in I_1}x_i - \sum_{i \in I_2} x_i$ and $\abs{y} = \sum_{j \in J_1} y_j - \sum_{j \in J_2} y_j$.
Because $\sum_{i \in I}x_i = 0$, we also have $\abs{x} = 2\sum_{i \in I_1} \abs{x_i} = 2\sum_{i \in I_2} \abs{x_i}$.
It follows that
\begin{equation*}
\begin{array}{lcl}
\abs{x} - \abs{y} & 
	= & \sum_{i \in I_1}x_i - \sum_{i \in I_2}x_i 
	- \left(\sum_{j \in J_1} y_j - \sum_{j \in J_2} y_j \right)\\
\ & 	= & \sum_{j \in J} \left(\sum_{i \in I_1} m_{i,j}x_i
	- \sum_{i \in I_2} m_{i,j}x_i \right)\\
	 & & - \sum_{j \in J_1} 
	\left(\sum_{i \in I_1} m_{i,j}x_i + \sum_{i \in I_2} m_{i,j}x_i \right)
	 + \sum_{j \in J_2}
	 \left(\sum_{i \in I_1} m_{i,j}x_i + \sum_{i \in I_2} m_{i,j}x_i \right) \\
\ & 	= &2\sum_{i \in I_1, j \in J_2} m_{i,j}x_i - 2\sum_{i \in I_2, j \in J_1} m_{i,j}x_i \\
 & 	= &  2 \sum_{i \in I_1, j \in J_2}m_{i,j} \abs{x_i} + 2 \sum_{i \in I_2, j \in J_1} m_{i,j} \abs{x_i}\\

 & 	\ge &  2 \sum_{i \in I_1, j \in J_2}\epsilon \abs{x_i} + 2 \sum_{i \in I_2, j \in J_1} \epsilon \abs{x_i}\\
 & 	= &  (\sharp J_2)\cdot \epsilon \cdot 2\sum_{i \in I_1} \abs{x_i} + (\sharp J_1) \cdot \epsilon \cdot 2\sum_{i \in I_2} \abs{x_i}\\
 & 	= &  (\sharp J_2) \cdot \epsilon \cdot \abs{x} + (\sharp J_1)\cdot \epsilon \cdot \abs{x}
\quad  = \quad t \epsilon \abs{x}
\end{array}
\end{equation*}
Therefore, we have that $(1-t\epsilon)\abs{x} \ge \abs{y}$, as desired.

\end{proof}

The next theorem extends \cite[Proposition 3.3 (b)]{GM}.
In \cite[Theorem 4.11]{BKMS}, the next theorem is referred in its complete form in the context of Bratteli diagrams of finite rank.

\begin{thm}\label{thm:limited-combinatorics-unique-ergodicity}
Let $\Ccal$ be a system of circuits enumerated by $\Nonne$ that expresses all invariant measures.
Suppose that $\Ccal$ has the system of winding matrices $\seb \barM_n \sen_{n \bni}$.
Let $\barM_n$ be a $(d_{n+1},d_n)$-matrix for each $n \bni$.
Suppose that there exists a series of positive numbers $0 < \epsilon_n $ $(n \bni)$ such that 
$\epsilon_n \le \barM_n$ for all $n \bni$.
Suppose that $\prod_{i = n}^{\infty}(1-d_i \epsilon_i) = 0$ for all $n \bni$.
Then, $(X,f)$ is uniquely ergodic.
\end{thm}
\begin{proof}
It is enough to show that $\diam(\xi_{\infty,n}(\Delta_{\infty})) = 0$ for all $n \bni$.
Fix $n \bni$.
By Lemma \ref{lem:positive-matrix}, $\diam(\xi_{m,n}(\Delta_m)) \le (\diam(\Delta_m))\prod_{i = n}^{m-1}(1-d_i \epsilon_i)$.
Therefore, we have that $\diam(\xi_{m,n}(\Delta_m)) \to 0$ as $m \to \infty$.
By Lemma \ref{lem:from-infinity-to-n}, $\Delta_n \supseteq \xi_{n+1,n}(\Delta_{n+1}) \supseteq (\xi_{n+2,n})(\Delta_{n+2}) \supseteq \dotsb$, and $\bigcap_{m>n}\xi_{m,n}(\Delta_m) = \xi_{\infty,n}(\Deltainf)$.
Therefore, we obtain $\diam(\xi_{\infty,n}(\Delta_{\infty})) = 0$, as desired.
\end{proof}

\end{document}